\documentclass[12pt]{article}

\usepackage{amsmath}
\usepackage{amsfonts}
\usepackage{amssymb}
\usepackage{extarrows}
\usepackage{amscd}
\usepackage{amsthm}
\usepackage{textcomp}
\usepackage{bbm}
\usepackage{color}
\usepackage[linktocpage=true,plainpages=false,pdfpagelabels=false]{hyperref}
\usepackage{euscript}
\usepackage{tikz}

\input xypic

\usepackage[matrix,arrow,curve]{xy}

\DeclareRobustCommand\longtwoheadrightarrow{\relbar\joinrel\twoheadrightarrow}

\usepackage{color}
\def\r{\color{red}}

\newcommand{\quash}[1]{}

\quash{
\sloppy
\oddsidemargin=1pt
\textwidth=450pt \textheight=640pt
\topmargin=-12mm
}


\numberwithin{equation}{section}

\newtheorem{defin}{Definition}[section]
\newtheorem{prop}{Proposition}[section]
\newtheorem{nt}{Remark}[section]
\newtheorem{Th}{Theorem}[section]
\newtheorem{lemma}{Lemma}[section]

\newtheorem{defin-prop}{Definition-proposition}[section]

\newfont{\ssdbl}{msbm8}
\newfont{\sdbl}{msbm9}
\newfont{\dbl}{msbm10 at 12pt}

\newcommand{\oo}{{\cal O}}
\newcommand{\ff}{{\cal F}}

\newcommand{\g}{{\cal G}}

\newcommand{\res}{\mathop {\rm res}}
\newcommand{\Hom}{\mathop {\rm Hom}}

\newcommand{\Aut}{\mathop {\rm Aut}}

\newcommand{\nm}{\mathop {\rm Nm}}
\newcommand{\Spec}{\mathop {\rm Spec}}

\newcommand{\dz}{\mathbb{Z}}

\newcommand{\Z}{\dz}

\newcommand{\lrto}{\longrightarrow}

\def\Z{{\mathbb Z}}
\def\Q{{\mathbb Q}}

\newcommand{\vp}{\varphi}

\newcommand{\AutL}{{   {\mathcal Aut}^{\rm c, alg} ({\mathcal L} )}}

\newcommand{\GG} {{\mathbb G}}
\newcommand{\CC}{\mathop{\rm CC}}

\newcommand{\G}{{\mathcal G}}

\newcommand{\vpl}{{\mathbb V}_+}
\newcommand{\vmi}{{\mathbb V}_-}
\newcommand{\Nil}{{\mathop {\rm Nil}}}

\newcommand{\rk}{{\mathop{\rm rk}}}

\newcommand{\bOmega}{\boldsymbol{\Omega}}

\begin{document}

\author{
Denis V. Osipov
}

\title{Contou-Carr\`ere symbol, Deligne pairing and quintets    \thanks{This work was performed at the Steklov International Mathematical Center and supported by the Ministry of Science and Higher Education of the Russian Federation (agreement no. 075-15-2025-303).}}
\date{}

\maketitle

\begin{abstract}
We prove the reciprocity laws for a family of projective curves over an affine base scheme, where curves can be singular and reducible, and the base can be non-Noetherian. These reciprocity laws generalize the Weil reciprocity law for a projective curve over a field and the reciprocity law for the Contou-Carr\`ere symbol.
We apply the proven reciprocity laws  to describe the actions of certain central extensions of the group ind-scheme related with the formal punctured disc on certain line bundles on the moduli stacks of quintets. A quintet consists of  geometric data including a family of curves and a line bundle on this family.  The line bundles on the moduli stacks are constructed using the Deligne pairings of line bundles from quintets.

\end{abstract}


\section{Introduction}
Recall the Weil reciprocity law (see~\cite[Chapter~III, \S~1.4]{Se}) in the following form.
Let $Y$ be a smooth projective curve over an algebraically closed field $k$. Let $T \subset Y$ be a finite (possibly, empty) set of points, and $U = Y \setminus T$. Let $f$ and $g$ be from $k(Y)^*$ such that the supports of divisors $(f)$ and $(g)$  do not intersect on  $U$. We have in $k^*$
\begin{equation}   \label{Weil_rec}
g( (f) \mid_U) \, f((g) \mid_U )^{-1}= \prod_{y \in T} (f,g)_y  \, \mbox{,}
\end{equation}
where for any divisor $D = \sum_i a_i y_i$ on $U$ (with $a_i \in \dz$, $y_i \in U$) and $d \in k(C)^*$ such that the supports of  divisors $(d) \mid_U$ and $D$ do not intersect
$$d(D) = \prod_{i} d(y_i)^{a_i} \, \mbox{,} \qquad  \mbox{and the tame symbol}   \quad   (f,g)_y = (-1)^{\nu_y(f) \nu_y(g)} \left[  f^{\nu_y(g)}  / g^{\nu_y(f)}   \right](y)  \, \mbox{,} $$
where $\nu_y$ is the valuation at a point $y$.

In Theorem~\ref{Weil-gen} we generalize the reciprocity law~\eqref{Weil_rec} to the case of a family of curves, i.e.  to a morphism  $\pi:   C \to \Spec A$ which is
a flat, proper, finitely presented morphism of schemes such that any fiber  of  $\pi$ is geometrically reduced and equidimensional of dimension~$1$.
In this case  $d(D)$ is replaced with the application of the norm maps on relative effective Cartier divisors on $C$, and the tame symbol is replaced with the Contou-Carr\`ere symbol. The Contou-Carr\`ere symbol $\CC$ (see more in Secton~\ref{Contou}) is the bimultiplicative antysimmetric morphism
$$
\CC \; : \; L \GG_m \times L \GG_m \lrto \GG_m
$$
that is given by an explicit formula (see~\eqref{CC-form-1}). Here $L \GG_m$ is the group ind-scheme which gives the functions on the formal punctured disc, and it is represented by the group functor $A \mapsto L \GG_m (A)= A((t))^*$, see Section~\ref{sec-loop}. When a commutative ring $A$ is a field, the Contou-Carr\`ere symbol $\CC$ evaluated on $A$ coincides with the tame symbol.

The reciprocity law from Theorem~\ref{Weil-gen} leads to the definition of the Deligne pairing $\langle L  , M \rangle$ for $C \to \Spec A$ (see more in
Section~\ref{Deligne}),
where $L$ and $M$ are line bundles on~$C$ and $\langle L, M  \rangle$ is the line bundle on $\Spec A$. This definition is the generalization of the construction when $A=k$, see~\cite[\S~6.1]{D1}. In the latter case $\langle L, M \rangle$ is just the $1$-dimensional $k$-vector space generated by pairs $(l,m)$, where $l$ and $m$ are rational sections of $L$ and $M$ correspondingly such that the supports of divisors $(l)$ and $(m)$ do not intersect, modulo relations
$$
(fl, m) \sim f ((m)) (l,m )    \, \mbox{,}  \qquad  (l, gm) \sim g((l))(l,m)  \, \mbox{,}
$$
where corresponding $f,g \in k(C)^*$, and this definition is correct because of the reciprocity law~\eqref{Weil_rec}  (when $T$ is empty).

Let $\AutL$ be the group ind-scheme which the automorphism group of the formal punctured disc. This group ind-scheme  represents the group functor
$A \mapsto \AutL (A)$, where $\AutL (A)$ is the group  of all $A$-automorphisms of the $A$-algebra $A((t))$ that are homeomorphisms (in natural $t$-adic topology of $A((t))$), see Section~\ref{quint}. The group ind-scheme $\g$ is $L \GG_m \rtimes \AutL$ with respect to the natural action of $\AutL$ on~$L \GG_m$.

We apply also the reciprocity law from Theorem~\ref{Weil-gen} (when it corresponds to $\mid T\mid =1$ in~\eqref{Weil_rec}) to the construction of the actions of central extensions of the group ind-scheme $\g$ by $\GG_m$ on the line bundles $\langle  \mathfrak{F}, \mathfrak{F} \rangle$, $\langle  \mathfrak{F}, \boldsymbol{\Omega}\rangle$
and $\langle  \boldsymbol{\Omega}, \boldsymbol{\Omega} \rangle$ on the moduli stacks $\widetilde{\mathcal M}$ or $\widetilde{\mathcal M}_{\rm sm}$. We explain it now.

The moduli stacks $\widetilde{\mathcal M}$ or $\widetilde{\mathcal M}_{\rm sm}$ (in the Zariski topology) associate with every commutative ring $A$ the groupoids of quintets $\widetilde{\mathcal M}(A)$ or $\widetilde{\mathcal M}_{\rm sm}(A)$. (We define $\widetilde{\mathcal M}(A)$ and $\widetilde{\mathcal M}_{\rm sm}(A)$ for affine schemes $\Spec A$ and then, using descent data and morphisms of descent data in the Zariski topology, we  consider $\widetilde{\mathcal M}$ and $\widetilde{\mathcal M}_{\rm sm}$ as stacks over the category of all schemes.)  The groupoid $\widetilde{\mathcal M}(A)$ consists of quintets over $A$ that are geometric data: a morphism $\pi : C \to \Spec A$ (with the properties as we described above for the reciprocity laws), a section of $\pi$ such that $\pi$ is smooth near the image of the section,
a line bundle on $C$,  a relative formal parameter
at the section, a formal trivialization of the line bundle at the section, see Sections~\ref{quint} and~\ref{more_quint}.  The groupoid $\widetilde{\mathcal M}_{\rm sm}(A)$ consists of quintets from $\widetilde{\mathcal M}(A)$ such that $\pi$ is a smooth morphism.

We note that in complex analytic category an analog of functor $A \mapsto \widetilde{\mathcal M}_{\rm sm}(A)$
modulo isomorphisms  (and when all fibers of $\pi$ from a quintet are  curves of
genus $g$)  is represented by an infinite-dimensional complex manifold, see~\cite{ADKP}. Such
quintets are related with the infinite-dimensional Sato Grassmanian via the Krichever
map.

The group ind-scheme $\g$ naturally acts on $\widetilde{\mathcal M}$ and $\widetilde{\mathcal M}_{\rm sm}$, see Section~\ref{quint}. Informally speaking, an element  from $\g(A)$ reglues  a quintet over $A$ in a formal punctured neighbourhood of the image of the  section (from the quintet) in $C$.

The line bundle $\langle  \mathfrak{F}, \mathfrak{F} \rangle$ is defined on $\widetilde{\mathcal M}$ by association with every quintet over $  A$ that contains a line bundle $\mathcal F$ the line bundle $\langle \mathcal F , \mathcal F \rangle$ on $\Spec A$. Analogously the line bundles
$\langle  \mathfrak{F}, \boldsymbol{\Omega}\rangle$
and $\langle  \boldsymbol{\Omega}, \boldsymbol{\Omega} \rangle$ are defined on $\widetilde{\mathcal M}_{\rm sm}$
by association with every quintet over $ A$ that contains a line bundle $\mathcal F$ the line bundles $\langle \mathcal F , \Omega_{C/A} \rangle$
and $\langle \Omega_{C/A} , \Omega_{C/A} \rangle$ on $\Spec A$.

 There are distinct $2$-cocycles
  $
\langle \Lambda , \Lambda   \rangle  $, $ \langle \Lambda , \Omega   \rangle$,  $\langle \Omega , \Omega   \rangle $
  on the group ind-scheme $\g$ with coefficients in the group scheme $\GG_m$ (where $\g$ acts trivially on $\GG_m$) that are given by explicit formulas with the help of the Contou-Carr\`ere symbol $\CC$, see Secton~\ref{last} or~\cite[\S~7.2]{O2}.
 These $2$-cocycles define three corresponding central extensions of $\g$ by $\GG_m$ that are the group ind-schemes. In Theorem~\ref{last-th} we construct natural actions (by explicit formulas) of  these three group ind-schemes on the line bundles $\langle  \mathfrak{F}, \mathfrak{F} \rangle$,
 $\langle \mathfrak F , \boldsymbol{\Omega} \rangle$
and $\langle \boldsymbol{\Omega} , \boldsymbol{\Omega}\rangle$
 on the corresponding moduli stacks of quintets.

Theorem~\ref{last-th} is based on Theorem~\ref{Th-CC}, where we make the local calculations of the corresponding $2$-cocycle related with the Contou-Carr\`ere symbol $\CC$ under the action of $L \GG_m(A) \times
L \GG_m(A)$ on the Deligne pairing $\langle L, M\rangle$. For the proof of Theorem~\ref{Th-CC} we use Theorem~\ref{Weil-gen}. Besides, in Theorems~\ref{stack-th}, \ref{stack-th-2}, and item~\ref{iitt-3} of Theorem~\ref{last-th}  we describe also actions of another central extensions of group ind-schemes related with $\g$ on the line bundles constructed similar  to the line bundles described above.

Finally, we note that distinct $2$-cocycles
  $
\langle \Lambda , \Lambda   \rangle  $, $ \langle \Lambda , \Omega   \rangle$,  $\langle \Omega , \Omega   \rangle $
and quintets were used in~\cite{O2}, where the local Deligne-Riemann-Roch isomorphism was constructed for line bundles in relative dimension~$1$ (see also the original work~\cite{D1} on the Deligne-Riemann-Roch isomorphism, and Section~1.1 in \cite{O2}). Therefore in connection with~\cite{O2} and the Deligne-Riemann-Roch isomorphism, constructing of  actions of the group ind-schemes given as central extensions by these three $2$-cocycles on the line bundles $\langle  \mathfrak{F}, \mathfrak{F} \rangle$,
 $\langle \mathfrak F , \boldsymbol{\Omega} \rangle$
and $\langle \boldsymbol{\Omega}, \boldsymbol{\Omega} \rangle$ was one of the main motivations for this paper.

\medskip

 {\bf Organization of the paper.}
In Section~\ref{sec-loop} we recall the properties of the loop functor $L\GG_m$ of the group functor $\GG_m$.
In Section~\ref{Contou} we recall the Contou-Carr\`ere symbol $\CC$.
In Section~\ref{Sec-Tate} we give some properties for Tate $A$-modules. In Section~\ref{Sec-comm} we give some properties for homomorphisms to the Picard groupoid ${\mathcal P}ic^{\dz}(A)$. In Section~\ref{rec_l} we prove the reciprocity laws for families of curves.
In Section~\ref{Deligne} we define the Deligne pairing for the family of curves using the proven reciprocity laws.
In Section~\ref{quint} we recall (and prove the properties for) the quintets and the action of the group ind-scheme $\g$ on the moduli stack  of quintets.
In Section~\ref{more_quint} we introduce the moduli stack $\widetilde{\mathcal M}$. In Section~\ref{stack} we construct an action of a central extension of $L \GG_m$ on a line bundle on $\widetilde{\mathcal M}$ defined by Deligne pairings. In Section~\ref{del-aut} we consider an action of $\AutL(A)$ on the Deligne pairings related with quintets. In Section~\ref{last}  we construct actions of central extensions of~$\g$ on line bundles on various moduli stacks related with quintets.

\medskip

 {\bf Notation and terminology.}
We fix some notation and terminology in this paper.   By a commutative ring we mean a commutative associative unital ring.  By  $A$ we always denote a commutative ring.
For any commutative ring $B$, by $B^*$ we denote its group of invertible elements.
By a group functor (or commutative group functor) we mean a covariant functor from the category of commutative rings to the category of groups (or Abelian groups).
We usually evaluate such a functor on $A$.

\section{Loop functor and the  Contou-Carr\`ere symbol}
\label{sec-two}

We recall some definitions and statements about the loop functor of $\GG_m$ and the Contou-Carr\`ere symbol.

\subsection{Loop functor of $\GG_m$}   \label{sec-loop}

 Let $A((t))= A[[t]][t^{-1}]$ be the ring of Laurent series over $A$. On the ring $A((t))$ there is the natural topology with  the base of neighbourhoods of zero consisting of $A$-submodules
$U_n = t^n A[[t]]$, $n \in \dz$. This topology makes the ring $A((t))$ into a topological ring.

 Recall that  $\GG_m$ is a commutative group functor such that $\GG_m(A)= A^*$.
By $L \GG_m$ we denote the group loop functor of $\GG_m$. This means that
$$L \GG_m(A) = \GG_m(A((t)))= A((t))^* \mbox{.} $$

There is a canonical  isomorphism of commutative group functors (see~\cite[Lemma~1.3]{CC1}, \cite[Lemma~0.8]{CC2})
\begin{equation}  \label{md}
L \GG_m   \, \simeq \,  \underline{\dz} \times \GG_m \times \vpl  \times \vmi
 \mbox{,}
\end{equation}
and we describe now  the group functors in the right hand side and the corresponding embeddings to $L \GG_M$.

The group  $\underline{\dz}(A)$ is the group of locally constant $\dz$-valued functions on $\Spec A$. Any $\underline{n} \in \underline{\dz}(A)$ defines the decomposition $A = A_1 \times \ldots \times A_l$ such that
the function $\underline{n}$ restricted to any $\Spec A_i$ is the constant function with value $n_i \in \dz$.
Then $\underline{n}  \mapsto t^{\underline{n}} = t^{n_1} \times \ldots \times t^{n_l}$ gives  embedding $\underline{\dz} \to L \GG_m$
of the corresponding group functors in~\eqref{md}.

We will use also the following notation.  For any $\underline{n} \in \underline{\dz}(A)$ as above and any $a \in A^*$
we define  the element $a^{\underline{n}} = a^{n_1} \times \ldots \times a^{n_k}$ from $ A^*$.

The group embedding  $\GG_m(A)  \hookrightarrow L \GG_m(A)$ is given by mapping of an invertible element from $A$ to the series consisting of only a constant term.

The subgroups $\vpl(A)$ and $\vmi(A)$ of the group $L \GG_m(A) $  are defined as
\begin{gather*}
\vpl(A) = \left\{ \left. 1 + \sum_{l > 0} a_l t^l  \quad \right|  \quad   a_l \in A     \quad  \mbox{for any} \quad l > 0         \right\}   \\
\vmi(A) = \left\{   1 + \sum_{l < 0} a_l t^l \quad \left|  \quad  \sum_{l < 0} a_l t^l  \in A((t)) \, \mbox{,} \; a_l \in \Nil(A)    \right.        \quad \mbox{for any} \quad l< 0           \right\} \, \mbox{,}
\end{gather*}
where $\Nil(A)$ is the nil-radical of $A$, i.e. the set of all nilpotent elements of $A$.

\quash{We denote
\begin{equation}  \label{declg}
(L \GG_m)^0 =   \GG_m \times \vpl  \times \vmi  \, \mbox{.}
\end{equation}
Thus we obtain decomposition of group functors
\begin{equation}  \label{dec}
\underline{\dz} \times (L \GG_m)^0 \, \simeq \,  L \GG_m  \, \mbox{.}
\end{equation}}
By $\nu : L \GG_m  \to \underline{\dz}$ we denote the  morphism of group functors given by the corresponding  projection in decomposition~\eqref{md}.

\subsection{Contou-Carr\`{e}re symbol} \label{Contou}

We recall the definition of the Contou-Carr\`{e}re symbol (see~\cite{CC1}, \cite[\S~2.9]{D2}, \cite[\S~2]{OZ}):
$$
\CC \,  :  \, L\GG_m  \times L \GG_m \lrto \GG_m  \, \mbox{.}
$$

We define now a  free $A((t))$-module of rang $1$:
\begin{equation}  \label{form-ker}
\widetilde{\Omega}_{A((t))}  = \Omega_{A((t))}  / N \, \mbox{,}
\end{equation}
where $\Omega_{A((t))}$ is the $A((t))$-module of absolute K\"ahler differentials, and the  $A((t))$-submodule $N$ is generated by all elements
 $df -  f' dt$, where $f \in A((t))$ and $f'= \frac{\partial{f}}{\partial{t}}$. It is clear  that $N$ contains elements $da$, where $a \in A$, and  that $dt$ is a basis of the $A((t))$-module~$\widetilde{\Omega}_{A((t))}$.

We define the  residue
$$
\res \; : \;  \Omega_{A((t))}  \lrto \widetilde{\Omega}_{A((t))}   \lrto A
$$
 where the first map is the natural map, and the second map is the map $\sum_{i \in \dz} a_it^i dt   \mapsto a_{-1}$.

The  Contou-Carr\`{e}re symbol $\CC$ is a bimultiplicative antisymmetric morphism which  has the following additional properties. Let  ${f,g  \in A((t))^*}$.

\begin{enumerate}
\item  If $a \in A^*$, then ${\CC}(a, g ) = a^{\nu(g)}$.
\item   $\CC(t,t)=-1$.
\item  If $\Q \subset A$ and $ f \in \vpl(A) \times \vmi(A) $, then
\begin{equation}  \label{CC-exp-log-1}
\CC(f,g) = \exp \res \left(\log f \cdot \frac{dg}{g} \right)  \, \mbox{,}
\end{equation}
where  $\exp (x)$ and $\log(1+y)$ are the usual formal series, the series $\log$ in above formula  converges in the topology of $A(((t))$,
and application of series $\exp$ in above formula makes sense, because  $\res \left(\log f \cdot \frac{dg}{g} \right)  \in {\rm Nil }(A)$.
\end{enumerate}

After restriction to commutative $\Q$-algebras,  the Contou-Carr\`{e}re symbol $${\CC}_{\Q} \, : \,  {L\GG_m}_{\Q}  \times {L \GG_m}_{\Q}  \lrto {\GG_m}_{\Q} $$
 is uniquely defined by the above properties.

There is the following unique extension of ${\CC}_{\Q}$ to $\CC$.

For any $f, g \in A((t))^*$ there are  unique decompositions:
$$
f = \prod_{i < 0} (1 - a_it^i) \cdot a_0 \cdot t^{\nu(f)}  \cdot \prod_{i > 0} (1 - a_it^i) \,
\mbox{,}  \qquad
g = \prod_{j < 0} (1 - b_jt^j) \cdot b_0 \cdot t^{\nu(g)}  \cdot \prod_{j > 0} (1 - b_jt^j)   \, \mbox{,}
$$
where $a_0, b_0 \in A^*$, $a_i, b_j \in {\rm Nil}(A) $ when $i, j <0$, and the products over negative $i$ and over negative $j$ are finite products.
Now the Contou-Carr\`{e}re symbol is defined by  the following formula:
\begin{equation}  \label{CC-form-1}
\CC(f,g)= (-1)^{\nu(f)  \nu(g)}  \frac{a_0^{\nu(g)} \prod_{i > 0} \prod_{j>0 }
 \left(1 - a_i^{j /(i,j)} b_{-j}^{i/(i,j)} \right)^{(i,j)}  }{b_0^{\nu(f)}
 \prod_{i > 0} \prod_{j>0 }
 \left(1 - a_{-i}^{j /(i,j)} b_{j}^{i/(i,j)} \right)^{(i,j)}
  }  \, \mbox{,}
\end{equation}
where the products in  the numerator and denominator actually consist of a finite number of factors, therefore the formula makes sense. (Formula~\eqref{CC-form-1} can be obtained, using above properties and by application of formula~\eqref{CC-exp-log-1} to elements $1 - a_it^i$ and $1 - b_jt^j$.)

\quash{ We recall, see~\eqref{md}, that we have embeddings $\GG_m \times \vpl  \hookrightarrow L\GG_m$
and $\GG_m \times \vmi  \hookrightarrow L\GG_m$. From above formulas for   the  Contou-Carr\`{e}re symbol $\CC$ it follows that
\begin{equation}  \label{CC-ident}
\CC \mid_{(\GG_m \times \vpl) \times (\GG_m \times \vpl)} =1  \qquad \mbox{and} \qquad
\CC \mid_{(\GG_m \times \vmi) \times (\GG_m \times \vmi)} =1
   \mbox{.}
\end{equation}}

\begin{nt}  \em
Another explicit formula for the  Contou-Carr\`{e}re symbol will follow from Proposition~\ref{explicit} below (cf.~also with~\cite[\S~3.3, Remark (i)]{BBE}).
\end{nt}

\begin{nt}  \em
There are other approaches to define the Contou-Carr\`{e}re symbol (even  the higher-dimensional  Contou-Carr\`{e}re symbol) by explicit formulas, see~\cite{GO}.
\end{nt}

\section{Relative reciprocity laws}
We will give the reciprocity laws for the family of projective algebraic curves. These reciprocity laws will generalize the reciprocity law for the smooth family of projective algebraic curves proved in~\cite[\S~3.4]{BBE}.
In particular case the reciprocity law will use the Contou-Carr\`ere symbol.

\subsection{Tate $A$-modules}  \label{Sec-Tate}

Let $\pi : C \to \Spec A$ be a finitely presented  morphism of schemes such that the fibres of $\pi$
  are one-dimensional schemes, $D \subset C$ be a relative effective Cartier divisor (see~\cite[Section~056P]{Stacks}) such that $\pi |_D$ is proper. For example, $\pi |_D$ is proper when $\pi$ is proper.

\begin{prop}  \label{exact-seq}
For any integers $n_1 \le n_2 \le n_3$ there is a natural exact sequence of finite projective (or, in other words, finitely generated projective) $A$-modules:
\begin{multline}
0 \lrto H^0 \left(C, \oo_C (n_2D) / \oo_C(n_1D) \right) \lrto H^0 \left(C, \oo_C (n_3D) / \oo_C(n_1D) \right)  \lrto  \\ \lrto H^0 \left(C, \oo_C (n_3D) / \oo_C(n_2)D \right)  \lrto 0  \, \mbox{.}  \label{e-s}
\end{multline}
\end{prop}
\begin{proof}
Since $\pi |_D$ is proper with finite fibers, by \cite[Lemma 02LS]{Stacks}, $\pi |_D$ is finite, and, in particularly, affine.
Since the scheme $D$ is affine, by induction on $(n_2 -n_1)$ and  \cite[Lemma 01XB]{Stacks} we have that $H^1 \left(C, \oo_C (n_2D) / \oo_C(n_1D) \right)  =0$. Hence we derive exact sequence~\eqref{e-s}.
 Now since $\pi |_D$ is finite, flat and  of finite presentation, by \cite[Section 02K9, Lemma 02KB]{Stacks},  $H^0 (C, \oo_C / \oo_C (-D))$ is a finite   projective $A$-module.
Hence for any integer $n$ the $A$-module  {$ H^0(C, \oo_C (nD) / \oo_C((n-1)D)$} is finite projective. Therefore, by induction on $(n_3 -n_1)$,  we have from exact sequence~\eqref{e-s} that
the $A$-module $ H^0(C, \oo_C (n_3D) / \oo_C((n_1)D)$  is finite projective.
\end{proof}

We consider the $A$-algebra
\begin{equation}  \label{compl-o}
\hat{\oo}_{C,D} = \varprojlim_{n > 0}  H^0(C, \oo_C / \oo_C( - nD))   \, \mbox{.}
\end{equation}
 It is clear that $\hat{\oo}_{C,D}$ has a natural topology of projective limit that makes  $\hat{\oo}_{C,D}$ into a topological $A$-algebra with the discrete topology on $A$.   We consider also the topological $A$-algebra
$$
{\mathcal K}_{C, D} =   \varinjlim_{m  > 0}  \varprojlim_{n > 0}  H^0(C, \oo_C (mD) / \oo_C( - nD))  \,  \mbox{.}
$$
with the topology of inductive and projective limits.

\begin{nt}  \em
Recall from~\cite[\S~2.11]{BBE} and  \cite[\S~3.2]{Dr} that a projective $A$-module considered as a discrete topological $A$-module is called a discrete Tate $A$-module. The dual ${M^* = \Hom_A(M,A)}$ to a discrete Tate $A$-module $M$ is called a compact Tate $A$-module (with the topology, where the base of neighbourhoods of $0$ is formed by
annihilators of ﬁnite subsets in
$M$). An elementary (or special) Tate $A$-module is a topological $A$-module which is a direct sum of a discrete Tate  $A$-module and a compact Tate $A$-module.

\end{nt}

From Proposition~\ref{exact-seq} it follows that $\hat{\oo}_{C,D}$  is a compact Tate module and ${\mathcal K}_{C, D} / \hat{\oo}_{C,D}$ is a discrete projective  $A$-module. Therefore ${\mathcal K}_{C, D}$ is an elementary Tate $A$-module.

By definition, an $A$-submodule $N  \subset {\mathcal K}_{C, D}$ is a $c$-lattice if $N$ is an open compact Tate $A$-module (with the induced topology) such that ${\mathcal K}_{C, D}  /N$ is a projective $A$-module.
Thus, $\hat{\oo}_{C,D}$ is a $c$-lattice. Clearly, $c$-lattices form a base of neighbourhoods of $0$ in  ${\mathcal K}_{C, D}$. It is easy to see that if $N_1 \subset N_2$ are $c$-lattices in ${\mathcal K}_{C, D}$,
then $N_2/N_1$ is a finite projective $A$-module.

\subsection{Homomorphism to the Picard groupoid ${\mathcal P}ic^{\dz}(A)$}  \label{Sec-comm}

Recall (see, e.g., more in \cite[\S~2A-2B]{OZ0}, \cite[\S~5.2(i)]{OZ})
that a Picard groupoid is a symmetric monoidal
group-like groupoid. Speaking more informally, a Picard groupoid is  a symmetric monoidal category such
that every object in this category is invertible, as well as every morphism.

We are interested in  the Picard groupoid ${\mathcal P}ic^{\dz}(A)$.
An object in ${\mathcal P}ic^{\dz} (A)$ is a pair $(L,n)$ where $L$
is a projective $A$-module of rank $1$, and $n  \in \underline{\dz}$. The
morphism set $\Hom_{{\mathcal P}ic^{\dz} (A)}\left((L_1,n_1),(L_2,n_2)\right)$ is
empty unless $n_1=n_2$ and $L_1$ is isomorphic to $L_2$, and in this case, it is just the set of isomorphisms between $L_1$ and $L_2$.
The tensor product (i.e., the monoidal structure) ${\mathcal P}ic^{\dz}(A) \times  {\mathcal P}ic^{\dz}(A) \to {\mathcal P}ic^{\dz}(A)$ is given as
\[(L_1,n_1)\otimes(L_2,n_2)\longmapsto (L_1\otimes L_2, n_1+n_2) \, \mbox{.}\]
There is a natural associativity constraint that makes ${\mathcal P}ic^{\dz}(A)$ a monoidal groupoid.
The commutativity constraint in the category ${\mathcal P}ic^{\dz}(A)$
is the following:
\[c_{L_1,L_2}
\, : \, (L_1\otimes L_2,n_1+n_2)\simeq(L_2 \otimes L_1,n_2+n_1)  \, \mbox{,} \qquad
c_{L_1,L_2}(v\otimes w)=(-1)^{n_1n_2}w\otimes v  \, \mbox{.}\]

Let $N_1$ and $N_2$ be any two $c$-lattices in ${\mathcal K}_{C, D}$.
The
 projective $A$-module of rank $1$
 $$
 \Hom\nolimits_A \left(\bigwedge^{\rm max}(N_1/ N),  \, \bigwedge^{\rm max} N_2/ N ) \right)
$$
does not depend on the choice of a $c$-lattice $N$ with $N \subset N_1$, $N \subset N_2$ up to a unique isomorphism. We identify over all such $c$-lattices $N$ all these projective $A$-modules via the following
definition of the projective $A$-module of rank $1$:
$$
\bigwedge^{\rm max} \left( N_1 \mid N_2 \right) =   \varinjlim_N \Hom\nolimits_A \left(\bigwedge^{\rm max}( N_1 / N), \, \bigwedge^{\rm max}(N_2/ N) \right)  \, \mbox{.}
$$

We define now $\det(N_1 \mid N_2)   $ as an object of ${\mathcal P}ic^{\dz}(A)$ in the following way:
$$
\det(N_1 \mid N_2)  = \left( \bigwedge^{\rm max} \left( N_1 \mid N_2 \right)
   , \,  \rk(N_2/N) - \rk(N_1/N)       \right)  \, \mbox{.}
$$

For any $c$-lattices  $N_1, N_2, N_3$  in ${\mathcal K}_{C, D}$ we have a canonical isomorphism
\begin{equation}
\label{iso-tens}
\det( N_1 \mid N_2)  \otimes \det(N_2 \mid N_3) \, \xrightarrow{\sim}  \, \det( N_1 \mid N_3)
\end{equation}
that satisfies the associativity diagram for any four $c$-lattices  in ${\mathcal K}_{C, D}$.

Multiplication on the $A$-algebra $ {\mathcal K}_{C, D}$  by any element $g$ from the group $ {\mathcal K}_{C, D}^*$ of invertible elements of $ {\mathcal K}_{C, D}$
is a homeomorphism.
Hence for any $c$-lattice $N  \subset {\mathcal K}_{C, D}$, the $A$-submodule $g N \subset {\mathcal K}_{C, D}$ is a $c$-lattice. Besides there is a canonical
  isomorphism
\begin{equation} \label{mult-isom}
\det(N_1 \mid N_2) \,  \xrightarrow{\sim}  \, \det(g N_1 \mid  g N_2) \mbox{.}
\end{equation}

\medskip

Recall (see more in~\cite[\S~2D]{OZ0})
that for a group $G$ and a Picard groupoid $\mathcal P$ one defines  the Picard groupoid of homomorphisms from $G$ to $\mathcal P$. The objects of this Picard groupoid are monoidal functors from
$G$ to  $\mathcal P$, where
$G$
is regarded as a
discrete monoidal category (the monoidal groupoid whose objects are elements of
$G$
and whose only morphisms are the unit morphisms of objects), and morphisms
between these monoidal functors are monoidal natural transformations.

Let $F$ be a homomorphism from $G$ to $\mathcal P$,  $Z_2 \subset G \times G $ be the subset of commuting elements, and $e$ is a unit object of $\mathcal P$.  One defines (see more in~\cite[\S~2D]{OZ0} )
the
anti-symmetric bimultiplicative map
$\mathop{\rm Comm}(F):Z_2\to\pi_1(\mathcal P)=\Aut_{\mathcal P}(e)$
in the following way.  For $g_1,g_2\in Z_2$  we consider the chain of isomorphisms
\begin{equation}  \label{comm}
F(g_1g_2)\simeq F(g_1)+F(g_2)\simeq F(g_2)+F(g_1)\simeq F(g_2g_1)=F(g_1g_2)  \, \mbox{,}
\end{equation}
where  $+$ denotes the monoidal structure in $\mathcal P$, the first and the third isomorphisms come from the
constraints for the homomorphism $F$ (i.e. that $F$ is a monoidal functor), and the second isomorphism comes from the commutativity
constraints of the Picard groupoid~$\mathcal P$. Using chain of isomorphisms~\eqref{comm} we  obtain an
element
\[\mathop{ \rm Comm}(F)(g_1,g_2)\in \Aut\nolimits_{\mathcal P}(F(g_1g_2))\simeq\pi_1({\mathcal P}).\]

\medskip

Now we construct the homomorphism ${\mathcal D}et_{C, D}$ from the group $ {\mathcal K}_{C, D}^*$ to the Picard groupoid ${\mathcal P}ic^{\dz}(A)$.
For any $g \in  {\mathcal K}_{C, D}^*$ we define an object from ${\mathcal P}ic^{\dz}(A)$:
\begin{equation}   \label{hom-det}
{\mathcal D}et_{C, D}(g) = \det(\hat{\oo}_{C,D}  \mid g \hat{\oo}_{C,D} ) \, \mbox{.}
\end{equation}
The structure of the monoidal functor for ${\mathcal D}et_{C, D}$ (i.e. that it is a homomorphism) comes from the following chain of isomorphisms constructed with the help of~\eqref{iso-tens}-\eqref{mult-isom}:
\begin{multline}
{\mathcal D}et_{C, D}(g_1 g_2) = \det(\hat{\oo}_{C,D}  \mid g_1 g_2 \hat{\oo}_{C,D} )  \simeq  \det(\hat{\oo}_{C,D}  \mid g_1  \hat{\oo}_{C,D} )  \otimes  \det( g_1 \hat{\oo}_{C,D}  \mid g_1 g_2 \hat{\oo}_{C,D} )  \simeq  \\
\simeq \det(\hat{\oo}_{C,D}  \mid g_1  \hat{\oo}_{C,D} )  \otimes  \det( \hat{\oo}_{C,D}  \mid  g_2 \hat{\oo}_{C,D} )  = {\mathcal D}et_{C, D}(g_1) \otimes {\mathcal D}et_{C, D}(g_2)   \label{det-gr}
\end{multline}
that satisfies the associativity diagram for any $g_1, g_2, g_3$.

Besides we have that ${ \rm Comm}({\mathcal D}et_{C, E})(g_1,g_2)  \in A^*$  for any $g_1, g_2$ from $ {\mathcal K}_{C, D}^*$.

\begin{prop}  \label{explicit} Let $f,g \in {\mathcal K}_{C, D}^*$ be any elements.
\begin{enumerate}
\item \label{it-1} If $f \in \hat{\oo}_{C,D}^*$ and $g \in {\mathcal K}_{C, D}^* \cap \hat{\oo}_{C,D}$, then $${ \rm Comm}({\mathcal D}et_{C, D})(f,g) = \det \left( f \mid_{\hat{\oo}_{C,D} / g \hat{\oo}_{C,D}}  \right) \, \mbox{.}$$
    (We recall that the determinant of the automorphism of the zero module equals $1$.)
\item  \label{it-2} If the $A$-algebra ${\mathcal K}_{C, E}$ is topologically isomorphic to the $A$-algebra $A((t))$, then
$${ \rm Comm}({\mathcal D}et_{C, D})(f,g) = \CC(f,g) \, \mbox{.}
$$
\end{enumerate}
\end{prop}
\begin{proof}
$\it 1 .$ The $A$-module $\hat{\oo}_{C,D} / g \hat{\oo}_{C,D}$ is finite projective, and $f$ acts as an automorphism of this $A$-module. Therefore $\det \left( f \mid_{\hat{\oo}_{C,D} / g \hat{\oo}_{C,D}}  \right)$ makes sense.
Now the statement follows directly from~\eqref{comm}-\eqref{det-gr}, using~\eqref{mult-isom}, since ${\mathcal D}et_{C, D}(f) \simeq(A, 0)$.

\noindent
$\it 2.$ This item follows from \cite[\S~3.2 - \S~3.3]{BBE}.
\end{proof}

\quash{
\begin{nt}  \em
The conditions of item~\ref{it-2} of Proposition~\ref{explicit} is satisfied, if $D$ is a section $\Spec A  \to C$, and
\end{nt}
}

\begin{nt}  \label{funct}
\em
Clearly, the constructions of Sections~\ref{Sec-Tate}-\ref{Sec-comm}, and, in particularly, the expression ${ \rm Comm}({\mathcal D}et_{C, D})(f,g)$,  are functorial
with respect to the ring $A$, i.e. with respect to ring homomorphisms $A \to A'$.
\end{nt}

\quash{
{\r To write on formula for $\CC$ that generalizes $\CC(a, g)$ from previous section, when $f \in \vpl(A)$, and $g \in \left( ( \vmi(A) \times \underline{\dz} ) \cap A[[t]] \right) \times \GG_m(A) \times \vpl(A)   $.
It gives the new definition of $\CC$ when we know the triviality of $\CC$ on some subgroups.}
}

\subsection{Reciprocity laws}  \label{rec_l}

We will use notation and assumptions from Sections~\ref{Sec-Tate} and~\ref{Sec-comm} and suppose additionally that $\pi$ is {\em a flat morphism}.

Recall that for any relative effective Cartier divisors $D_1$ and $D_2$ on $C$ the ideal sheaf of the relative effective Cartier divisor $D_1+D_2$ is defined by the product of the ideal sheaves of $D_1$ and $D_2$,
see~\cite[Lemma 0B8U]{Stacks}.

We will say that    effective Cartier divisors $D_1$ and $D_2$ on $C$ are disjoint if the closed sets which are supports of  $D_1$ and $D_2$ do not intersect.

By definition,
for any open set $U \subset C$  the set ${\mathcal S}_{C/A}(U) \subset \oo_C(U)$ is the set   of relative regular elements from $\oo_C(U)$, i.e. the elements that are regular after restriction to each fiber of $\pi |_U$.
(Recall that a regular element is, by definition, a non-zero divisor. Besides,  any element from ${\mathcal S}_{C/A}(U)$  is also a regular element from $\oo_C(U$) and defines the relative effective Cartier divisor on $U$ for $\pi |_U$, see \cite[Lemma 062Y]{Stacks}.)

Now,  by definition, {\em the sheaf of relative meromorphic functions} ${\mathcal R}_{C/A}$ is the sheaf on~$C$ associated with a presheaf of rings
$
U \mapsto {\mathcal S}^{-1}_{C/A} (U) \oo_C(U)
$. (We will use also the notation ${\mathcal R}_{C/S}$ for ${\mathcal R}_{C/A}$, when $S = \Spec A$.)

We have the natural embedding  of $\oo_C$ into $ {\mathcal R}_{C/A}$. Therefore for any relative effective Cartier divisor $D \subset C$ we have a natural embedding of invertible ideal sheaf $\oo_C(-D)$ of $D$ into $ {\mathcal R}_{C/A}$, and hence a natural embedding of the invertible sheaf ${\oo_C(D) = \oo_C(-D)^{\otimes -1}}$ into ${\mathcal R}_{C/A}$.

Let $D \subset C$ be a relative effective Cartier divisor such that $\pi |_D$ is proper.
By Proposition~\ref{exact-seq}, $\pi_* \oo_D$ is finite locally free. Therefore there is a norm map
$$
\nm\nolimits_{D/A} \, : \, \pi_* \oo_D \lrto \oo_{\Spec A}
$$
that maps a function
$h$ to the determinant of the multiplication by
$h$, which is an endomorphism of the $\oo_{\Spec A}$-module $\pi_* \oo_D$. And for any $f \in \oo_C(V)^*$, where an open set $V \subset C$ contains $D$, we set
$
f(D) = \nm\nolimits_{D/A}(f)  \, \in \, A^*
$. (We note that for the empty effective Cartier divisor $D$ the image of $\nm\nolimits_{D/A}$ is $1$.)

For any relative effective Cartier divisors $D_1$ and $D_2$ on $C$ such that  $\pi|_{D_1}$ and $\pi |_{D_2}$ are proper, and for any $f \in \oo_C(V)^*$,  where an open set $V \subset C$ contains $D_1$ and $D_2$,  we set
$f(D_1 - D_2) = f(D_1) f(D_2)^{-1}$.  For any other appropriate relative effective Cartier divisors $E_1$ and $E_2$ such that $\oo_C(D_1) \oo_C(-D_2) = \oo_C(E_1) \oo_C(-E_2)$ in $ {\mathcal R}_{C/A}$,
i.e. ${D_1 + E_2 = D_2 + E_1}$,
 we have $f(D_1 -D_2) = f(E_1 -E_2)$, since ${f(D_1 + E_2) = f(D_1) f(E_2)}$.

Let ${\mathcal R}_{C/A}(C)^{*}$ be the group of invertible elements in the ring ${\mathcal R}_{C/A}(C)$.
For  ${g \in  {\mathcal R}_{C/A}(C)^{*}}$ we will say that $(g) = D_1 - D_2$, where $D_1$ and $D_2$ are any relative effective Cartier divisors on $C$, when $g \oo_C = \oo_C(D_2) \oo_C(-D_1)$  in $ {\mathcal R}_{C/A}$.  (This is exactly the case when $g$ comes from $\oo_C(W)^*$ and its embedding  into $   {\mathcal R}_{C/A}(C)^{*}$, where an open set $W$ is the complement to a relative effective Cartier divisor in~$C$.) For any ${f \in \oo_C(V)^*}$,  where an open set $V \subset C$ contains $D_1$ and $D_2$, we set
{${f((g)) = f(D_1 -D_2)}$}.

Let $D \subset C$ be a relative  effective Cartier divisor, $U \subset C$ be an open subset that contains $D$. Then there is a natural homomorphism of $A$-algebras: $\oo_C(U \setminus D) \to {\mathcal K}_{C, D}$.

\begin{Th}  \label{Weil-gen}
Let $\pi : C \to \Spec A$ be a flat, proper, finitely presented morphism of schemes such that any fiber  of  $\pi$ is equidimensional of dimension~$1$.
 \begin{enumerate}
 \item \label{general} Let $D \subset C$ be a relative effective Cartier divisor such that the support of  $D$ intersects every irreducible component of $\pi^{-1}(s)$  for any $s \in \Spec A$.
 Let ${f, g \in \oo_C(C \setminus D)^*}$. We have
 \begin{equation}   \label{commd}
 { \rm Comm}({\mathcal D}et_{C, D})(f,g)  =1  \, \mbox{.}
 \end{equation}
 \item \label{ii-2} Under conditions of item~\ref{general} suppose additionally that $D= D_1 + \ldots + D_n$, where  $D_i $ are relative effective Cartier divisors, $D_i$ and $D_j$ are disjoint for any $i \ne j$.
 We have
 \begin{equation}   \label{commd-prod}
 \prod_{1 \le i \le n}  { \rm Comm}({\mathcal D}et_{C, D_i})(f,g)  =  { \rm Comm}({\mathcal D}et_{C, D})(f,g)  =1   \, \mbox{.}
 \end{equation}
 \item  \label{ii-3}
 Let $f ,g \in  {\mathcal R}_{C/A}(C)^{*}$ such that $(f)= D_1 - D_2$ and $(g) = E_1 -E_2$, where $D_1, D_2, E_1, E_2$ are relative effective Cartier divisors. Suppose that the divisors ${D_1 +D_2}$ and ${E_1 +E_2}$ are disjoint,
 and the support of the divisor $D_1 +D_2 + E_1 +E_2$  intersects every irreducible component of $\pi^{-1}(s)$  for any $s \in \Spec A$. We have
 $$
 f ((g))  = g((f))   \, \mbox{.}
 $$
 \item  \label{ii-4}  Let $H= H_1 + \ldots + H_n$, where   $H_i$  is a relative effective Cartier divisors and the $A$-algebra ${\mathcal K}_{C, H_i}$ is topologically isomorphic to the $A$-algebra $A((t_i))$ for any $i$. Let  $D_1, D_2, E_1, E_2$ be relative effective Cartier divisors such that any two of divisors $D_1 +D_2$, $E_1 +E_2$, $H_1, \ldots, H_n $ are disjoint.  Let $J = D_1 + D_2 +E_1 + E_2 + H$  and $f,g  \in \oo_C(C \setminus J)^*$ such that  on $C \setminus H$ the following condition is satisfied
 $$
 (f |_{C \setminus H}) = D_1 - D_2  \, \mbox{,} \qquad \qquad (g |_{C \setminus H}) = E_1 -E_2 \, \mbox{.}
 $$
 Suppose that the support of  $J$  intersects every irreducible component of $\pi^{-1}(s)$  for any $s \in \Spec A$. We have
 $$
  g(D_1 -D_2) f(E_1 -E_2)^{-1}= \prod_{1 \le i \le n} \CC\nolimits_{H_i}(f,g)  \, \mbox{,}
 $$
 where $\CC_{H_i}$ is the  Contou-Carr\`{e}re symbol constructed by ${\mathcal K}_{C, H_i}$.
 \item  \label{ii-last} Let  $\pi^{-1}(s)$ be geometrically reduced for any $s \in \Spec A$. Then in the previous items of this theorem we do not need the conditions on intersections of the supports of relative effective Cartier divisors with the fibers $\pi^{-1}(s)$.
 \end{enumerate}
\end{Th}
\begin{proof}
$\it 1.$
By~\cite[Chapter~7, Prop.~5.5, Exerc.~5.4]{Liu} the restriction of   $D$ to   ${\pi^{-1}(s)}$ is an ample Cartier divisor on ${\pi^{-1}(s)}$ for any $s \in \Spec A$.
Therefore from \cite[Corollaire 9.6.4]{EGA-IV-3} it follows that  $\oo_{C}(D)$ is an ample invertible sheaf on $C$.  Now arguing as in Proposition~6 of~\cite{O2} (see the corresponding references in the proof of this proposition) and using Proposition~\ref{exact-seq} we have that   there is an integer $n > 0$ such that for any integers $q > 0$ and any $m \ge n$ $H^q(C, \oo_C(mD)) =0 $ and $H^0(C , \oo_C(mD))$ is a finite projective $A$-module. Besides, for an integer $r \le m$ the cohomology groups of the complex
\begin{equation}   \label{compl-AG}
H^0(C, \oo_C(mD))  \lrto H^0(C, \oo_C(mD)/\oo_C(rD) )
\end{equation}
coincide with the cohomology groups $H^*(C, \oo_C(rD))$.  Then    complex~\eqref{compl-AG} is naturally quasi-isomorphic (via embedding of complexes) to the complex
\begin{equation}  \label{compl2}
\oo_C(C \setminus D)   \lrto {\mathcal K}_{C, D}/\hat{\oo}_{C,D,r}    \, \mbox{,}
\end{equation}
where $\hat{\oo}_{C,D,r}  = \varprojlim_{t > 0}  H^0(C, \oo_C(rD) / \oo_C( (r- t)D))$.

Now for any $g \in \oo_C(C \setminus D)^* $  the complex
\begin{equation}  \label{compl3}
\oo_C(C \setminus D)   \lrto {\mathcal K}_{C, D}/g\hat{\oo}_{C,D,r}    \, \mbox{.}
\end{equation}
is isomorphic to complex~\eqref{compl2}  via multiplication by $g^{-1}$.
Hence and from complexes~\eqref{compl-AG}-\eqref{compl2} it is easy to see that there is an integer ${k > n}$ such that for any integer ${l \ge k}$  complex~\eqref{compl3}
is naturally quasi-isomorphic to the complex of finite projective $A$-modules
\begin{equation} \label{compl-det}
H^0(C, \oo_C(lD)) \lrto  \hat{\oo}_{C,D,l} /g\hat{\oo}_{C,D,r}  \, \mbox{.}
\end{equation}
Now we define an object $\det({\mathcal H}^*(g\hat{\oo}_{C,D,r}) )$ from
${\mathcal P}ic^{\dz}(A)$ as the determinant of the complex~\eqref{compl-det}, using also the  rank difference as an element from $\underline{\dz}$. (Clearly, these determinants are canonically isomorphic for various $l$, so we can identify the corresponding projective $A$-modules of rank $1$  by taking the inductive limit with respect to $l$ similar as we did in Section~\ref{Sec-comm} for the definition of $\det(  \cdot \mid \cdot)$.)

It is easy to see that for any integers $r_1, r_2$ and elements $g_1, g_2 \in \oo_C(C \setminus D)^* $ there is a canonical isomorphism
\begin{equation}   \label{coh-isom}
\det ({\mathcal H}^*(g_1\hat{\oo}_{C,D,r_1}) )  \otimes \det( g_1\hat{\oo}_{C,D,r_1} \mid  g_2\hat{\oo}_{C,D,r_2} ) \,  \xrightarrow{\sim} \,  \det ({\mathcal H}^*(g_2\hat{\oo}_{C,D,r_2}) )
\end{equation}
that satisfies the associativity diagram  with the help  of  isomorphism~\eqref{iso-tens}   if we add also  $\det( g_2\hat{\oo}_{C,D,r_2} \mid  g_3\hat{\oo}_{C,D,r_3} )$.
(First, isomorphism~\eqref{coh-isom}) is being constructed  when the $c$-lattice $g_1\hat{\oo}_{C,D,r_1}$ contains or is contained in the $c$-lattice $g_2\hat{\oo}_{C,D,r_2}$.)

Now for any $g \in \oo_C(C \setminus D)^*$  from~\eqref{hom-det} and~\eqref{coh-isom} we have a canonical isomorphism
\begin{equation}  \label{final-iso}
{\mathcal D}et_{C, D}(g)  \, \xrightarrow{\sim}  \, \det ({\mathcal H}^*(\hat{\oo}_{C,D}) )^{-1}  \otimes \det ({\mathcal H}^*(g \hat{\oo}_{C,D}) )
\end{equation}
Hence the homomorphism ${\mathcal D}et_{C, D}$ restricted to the subgroup $\oo_C(C \setminus D)^*   \subset {\mathcal K}^{*}_{C, D}$ is isomorphic to the trivial homomorphism, since the multiplication by $g$
induces the isomorphism between $\det ({\mathcal H}^*(\hat{\oo}_{C,D}) )$  and  $\det ({\mathcal H}^*(g \hat{\oo}_{C,D}) )$, and this isomorphism gives the isomorphism between the right hand side of~\eqref{final-iso}
and $(A, 0)$ that we need. Therefore from~\cite[Remark~2.7, Corollary~2.10]{OZ0} we have~\eqref{commd}.

\noindent  $\it 2.$ Since $D_i$ and $D_j$ are disjoint for any $i \ne j$, we have that
$$
 {\mathcal K}_{C, D} \simeq \prod_{1 \le i \le n }   {\mathcal K}_{C, D_i}  \, \mbox{.}
$$
Hence the homomorphism ${\mathcal D}et_{C, D}$  is isomorphic to the sum of homomorphisms ${\mathcal D}et_{C, D_1}$, $\ldots$, ${\mathcal D}et_{C, D_n}$  (cf.~\cite[Lemma~4.3]{OZ0}). Therefore by~\cite[Corollary~2.9]{OZ0} and item~\ref{general} of this theorem
we have~\eqref{commd-prod}.

\noindent $\it 3.$ This item follows from item~\ref{ii-2} of this theorem and item~\ref{it-1} of Proposition~\ref{explicit}.

\noindent $\it 4.$  This item follows from item~\ref{ii-2} of this theorem and  Proposition~\ref{explicit}.

\noindent $\it 5.$
The statements of items~\ref{ii-2}--\ref{ii-4} are about equalities of elements from $A^*$.
We have functoriality from Remark~\ref{funct}. Therefore it is enough to prove that for any point ${s \in \Spec A}$ there is an \'{e}tale neighbourhood $W$ of $s$ such that the statements of items~\ref{ii-2}--\ref{ii-4} are valid after restriction to this neighbourhood.

We will find (see below) an \'{e}tale neighbourhood $\widetilde{W}$ of $s$, where there is an additional relative effective  Cartier divisor $D$ on $  C_{\widetilde{W}} = C \times_{\Spec A} \widetilde{W}$ such that $D$ is disjoint with the pullbacks in $ C_{\widetilde{W}}$ of
all the other relative effective Cartier divisors that appear in items~\ref{ii-2}--\ref{ii-4} and the support of $D$ intersects every irreducible component of the pullback $\widetilde{\pi^{-1}(s)}$ of $\pi^{-1}(s)$ in $C_{\widetilde{W}}$. Then by \cite[Chapter~7, Prop.~5.5, Exerc.~5.4]{Liu} the restriction of   $D + K$ to   $\widetilde{\pi^{-1}(s)}$ is an ample Cartier divisor on $\widetilde{\pi^{-1}(s)}$, where $K$ is any relative effective Cartier divisor on $C_{\widetilde{W}}$.   Therefore by \cite[Corollaire 9.6.4]{EGA-IV-3} there is an open neighbourhood $W \subset \widetilde{W}$ of the preimage of $s$ in $\widetilde{W}$ such that $\oo_{C_{W}}(D+K  |_{C_W})$ is an ample invertible sheaf on $C_W = C \times_{\Spec A} W$. This is enough for the proof of item~\ref{general} (and other items) after addition the divisor $D$, because  the main ingredient in the proof is the ampleness of
the invertible sheaf that corresponds to
 a relative effective Cartier divisor, and for any open $U \subset C_W$ that contains $D |_{C_W}$ and for any $f,g \in \oo_{C_W}(U)^*$ by item~\ref{it-1} of Proposition~\ref{explicit} we have ${ \rm Comm}\left({\mathcal D}et_{C_W, D  |_{C_W}} \right) \left(f,g \right)  =1$.

Now we construct such \'{e}tale neighbourhood $\widetilde{W}$ of $s$. By~\cite[Lemma~056V]{Stacks}, the fiber $\pi^{-1}(s)$ contains a dense open subset  which is smooth over the residue field $k(s)$ of $s$.
Take a smooth point on the fiber. Since $\pi$ is flat, by  \cite[Lemma~01V9]{Stacks} $\pi$ is smooth in a neighbourhood $V$  of this point on $C$. By \cite[Lemma~055U]{Stacks}  there is an  \'{e}tale neighbourhood  $W'$ of $s$
and a section  $ \alpha :  W' \to V_{W'}$. Since $\pi $ and its base change are separated, the image of $\alpha$, which we denote by the same letter, is a closed subscheme of $C_{W'}$. Now by \cite[Theorem 17.12.1]{EGA4}
$\alpha$ is a (relative effective) Cartier divisor on $C_{W'}$.  Taking various smooth  points on $\pi^{-1}(s)$ over $k(s)$, and taking smaller open subsets in \'{e}tale neighbourhoods of $s$ such that the involved relative effective divisors would be disjoint, we construct  an  \'{e}tale neighbourhood $\widetilde{W}$ of $s$ and a relative effective  Cartier divisor $D$ on $  C_{\widetilde{W}}$ which we need.

\end{proof}

\quash{\begin{nt}  \em
Theorem~\ref{Weil-gen} is a generalization of the Weil reciprocity law for a projective curve (see~\cite[Chapter~III, \S~4]{Se}) to a relative case and when a fiber can be singular.
\end{nt}}

\section{Deligne pairing}  \label{Deligne}
We give the definition of the Deligne pairing  for a pair of line bundles on the family of projective  curves by means of the generalisation of Weil reciprocity law to the relative case as in~Theorem~\ref{Weil-gen}.

The Deligne pairing was introduced in~\cite[\S~6]{D1}, \cite[Expos\'e~XVIII, \S~1.3]{SGA4}  for the case of smooth proper morphism of relative dimension~$1$.  Later the Deligne construction was generalized to families of arbitrary dimension, i.e. to morphisms of relative dimension greater than $1$ between schemes, see~\cite{E}.   Note that in~\cite{E} it was assumed that a morphism is projective, flat and Cohen-Macaulay, and  schemes are Noetherian.

We will use the construction of the Deligne pairing via the generalization of the Weil reciprocity law that are similar to construction from~\cite[Chapter~XIII, \S~5]{ACG} with the difference that we will replace the complex analytic topology by the \'etale topology.

During this section we assume that {\em $\pi : C \to S$ is a flat, proper finitely presented morphism of arbitrary schemes such that
any fiber of $\pi$ is geometrically reduced and equidimensional of dimension $1$.}

Let $L$ be an  invertible sheaf  of  $\oo_C$-modules. By { \em a relative meromorphic section} $l$ of $L$ we mean an isomorphism of invertible $\oo_C$-modules
$$l \; : \; L \, \stackrel{\sim}{\lrto}  \, \oo_C(D_1) \oo_C(-D_2)  \subset {\mathcal R}_{C/S}  \, \mbox{,}$$
where ${\mathcal R}_{C/S}$  is the sheaf of relative meromorphic functions (see definitions in Section~\ref{rec_l},  where one has to change an affine scheme $\Spec A$ to an arbitrary scheme $S$), and  $D_1$ and $D_2$ are relative effective Cartier divisors on $C$. We denote the divisor  ${ (l)= D_1 - D_2 }$. By  ${\mathcal R}_{C/S}^*(L) $ we denote {\em the set of all relative meromorphic sections} of $L$.
We note that the subset ${\mathcal R}_{C/S}^*(\oo_C)  \subset {\mathcal R}_{C/S}(C)^* $ is a subgroup that consists of elements $h \in {\mathcal R}_{C/S}(C)^*$ such that $(h)= D_1 - D_2$, where $D_1$ and $D_2$ are relative effective Cartier divisors.

Suppose that ${ l_1, l_2 \in {\mathcal R}_{C/S}^*(L)}$. Then we have $l_1 = g \circ l_2$, where $g \in {\mathcal R}_{C/S}^*(\oo_C)  $ and  $g$ as element from  ${\mathcal R}_{C/S}(C)^*$ acts on  ${\mathcal R}_{C/S}$ by multiplication by $g$. We will write also that
$l_1 / l_2 = g$ and $l_1 = g l_2$. Let ${ (l_1)= E_1 - E_2 }$ and ${ (l_2)= F_1 - F_2 }$, where $E_1, E_2, F_1, F_2$ are relative effective Cartier divisors on $C$. We will say that the divisors
${ (l_1)}$ and ${ (l_2)}$ are disjoint if the effective Cartier divisors $E_1 + E_2$ and $F_1 + F_2$ are disjoint.

Suppose that $L$ and $M$ are two invertible sheaves of $\oo_C$-modules. Suppose that there is  a pair $(l,m)$, where  $l \in {\mathcal R}_{C/S}^*(L)$ and $m \in {\mathcal R}_{C/S}^*(M)$,
such that the divisors $(l)$ and $(m)$ are disjoint. Suppose that $S = \Spec A $.  We consider the $A$-module $\langle L , M \rangle$ called {\em the Deligne pairing} which is the quotient of the $A$-module  freely generated by all such pairs $(l,m)$
modulo the equivalence relation generated
by
$$
(fl, m) \sim f ((m)) (l,m )    \, \mbox{,}  \qquad  (l, gm) \sim g((l))(l,m)  \, \mbox{,}
$$
where corresponding $f,g \in {\mathcal R}_{C/S}^*(\oo_C)$. Denote the class of the pair $(l,m)$ in $\langle L , M \rangle$ by $\langle l , m \rangle$. From~items~\ref{ii-3} and~\ref{ii-last} of Theorem~\ref{Weil-gen}  it follows that $\langle L , M \rangle$ is a free $A$-module of rank $1$ (see also~\cite[Chapter~XIII, \S~5]{ACG}).

Now suppose that $S$ is an arbitrary scheme.
\begin{lemma}  \label{et-n}
For any point $s \in S$ there is an \'etale affine neighbourhood $U$ of  $s$ such that the pullbacks of $L$ and $M$ to $C_{U}$ have relative meromorphic sections $l$ and $m$  with disjoint divisors.
\end{lemma}
\begin{proof}
Similar to  the proof of item~\ref{ii-last} of Theorem~\ref{Weil-gen} we find an \'etale  affine neighbourhood $V $ of  $s$ and a relative effective Cartier divisor $D$ on $C_{V}$,
where $\pi_V : C_V \to V$ is the base change,
 such that $\oo_{C_{V}}(D)$ is an ample invertible sheaf. Let $\widetilde{L}$ be the pullback of $L$ to $C_V$.    Now   there is an integer $n >0 $ such that for any integer $k >n $  we have
  $R^q (\pi_V)_* \left(\widetilde{L} (kD) \right) = 0$ for any $q >0$ and
  $(\pi_V)_* \left(\widetilde{L} (kD) \right)$
   is a finite locally  free sheaf,  and these sheaves commute with any base change,   see the proof of Proposition~6 from~\cite{O2} and also the references therein, and also~\cite[Chapter~5, Prop.~2.28]{Liu}.

   Hence, taking a smaller open subset $\widetilde{U}$ of $V$ that contains the preimage of $s$ in $V$,
   we construct a relative meromorphic section $l = l' \otimes  1_{kD}^{\otimes -1} $ of  the pullback of $L$ to $C_{\widetilde{U}}$, where $k >n$ and $k$ is big enough,    $l'$ is a relative section of the pullback of $\widetilde{L}(kD)$ to $C_{\widetilde{U}}$, and $1_{kD}$ is the relative section of the pullback of $\oo_{C_{{U}}}(k {D})$  to $C_{\widetilde{U}}$ induced by $1 \in H^0 \left(C_{\widetilde{U}}, \oo_{C_{\widetilde{U}}}\right)$. (A section will be a relative section when it does not vanish identically on any irreducible component of the fiber over any point of $\widetilde{U}$,   see~\cite[Lemma~062Y]{Stacks}).

   Analogously, but with the help of another $D$,  we construct a relative meromorphic section $m  $ of  the pullback of $M$ to $C_{U}$ such that the preimage of the divisor of $l$ in $C_U$ is disjoint from the divisor of $m$, where $U$ is an  \'etale affine neighbourhood of  a point $\widetilde{s} \in \widetilde{U}$ that is mapped to $s$. (Here,  for the construction, it is enough first to state that the preimages of divisors of $l$ and $m$ in the preimage of $C_{\widetilde{s}}$ are disjoint.)
   \end{proof}

Lemma~\ref{et-n} and the canonical isomorphism $H^1\left(S, \oo_S^*\right)  \to H^1_{etale}\left(S, \GG_m\right)$ (called Hilbert $90$, see~\cite[Theorem~03P8]{Stacks}) show that the following definition is correct.

\begin{defin}
Let $L$ and $M$ be invertible sheaves of $\oo_C$-modules, where $\pi : C \to S$ as in the beginning of this section. Let $\bar{s}  \to S$ be a geometric point and $\oo_{S, \bar{s}}$ be the local ring at $\bar{s}$ for the \'etale topology.
For the base change scheme $C_{\Spec \oo_{S, \bar{s}}} $ consider  a rank $1$ free $\oo_{S, \bar{s}}$-module $\langle L, M \rangle_{\bar{s}}$ as above. We now define an invertible sheaf $\langle L, M \rangle$ of $\oo_S$-modules
called the Deligne pairing
 by describing its sections over \'etale morphisms.
Given an \'etale morphism $U \to S$, a collection $\{ u_{\bar{s}} \in  \langle L, M  \rangle_{\bar{s}} :  \bar{s}  \to U  \}$ is a section of the pullback of $\langle L, M  \rangle$ over $U$ if and only if, for every $\bar{s}  \to U$,
there is an affine \'etale neighbourhood $U'$ of $\bar{s}$ and relative meromorphic sections $l$ of the pullback $L$ and $m$ of the pullback $M$ over $C_{U'}$ such that $u_t = \langle l,m \rangle$ for every geometric point $\bar{t}  \to U'$.
\end{defin}

The following proposition follows directly from the definition of the Deligne pairing.
\begin{prop}   \label{prop-Del-pair}
\begin{enumerate}
\item The Deligne pairing and the isomorphisms in items below are compatible with  base change and functorial with respect to isomorphisms of  sheaves.
\item \label{itt-1} Given invertible sheaves $L_1$, $L_2$, $M$ of $\oo_C$-modules there are well-defined isomorphisms
$$
\langle L_1, M  \rangle  \otimes_{\oo_S} \langle L_2, M  \rangle \stackrel{\sim}{\lrto}  \langle L_1 \otimes_{\oo_C} L_2, M  \rangle  \mbox{,}  \quad
  \langle M, L_1  \rangle  \otimes_{\oo_S} \langle M, L_2  \rangle \stackrel{\sim}{\lrto}  \langle M , L_1 \otimes_{\oo_{C}} L_2   \rangle
$$
which satisfy the  associativity conditions and are given (\'etale locally) by
$$
 \langle l_1, m  \rangle  \otimes \langle l_2, m  \rangle \stackrel{\sim}{\lrto}  \langle l_1 \otimes  l_2, m  \rangle \mbox{,}  \qquad
\langle m, l_1  \rangle  \otimes \langle m, l_2  \rangle \stackrel{\sim}{\lrto}  \langle m , l_1 \otimes l_2 \rangle   \mbox{.}
$$
\item   \label{iii-3} Given two invertible sheaves  $L$
and $M$ of $\oo_C$-modules, there is a canonical
isomorphism $\tau : \langle L , M  \rangle  \to \langle M, L  \rangle  $, where $\langle l, m  \rangle  \mapsto \langle m, l \rangle $, which is compatible with the isomorphisms in item~\ref{itt-1}.
\end{enumerate}
\end{prop}

\begin{nt} \label{Gen-det} \em
By the same method as in~\cite[Chapter~XIII, Theorem~(5.8)]{ACG} by changing the complex analytic topology to the \'etale topology one can prove the following canonical isomorphism which is compatible with  base change:
$$
\langle L , M  \rangle   \stackrel{\sim}{\lrto}  \det R\pi_* ( L \otimes_{\oo_C} M) \otimes_{\oo_S}  \det R\pi_* ( L)^{\otimes -1}  \otimes_{\oo_S}  \det R\pi_* ( M)^{\otimes -1}  \otimes_{\oo_S}  \det R\pi_* ( \oo_C)  \, \mbox{.}
$$
Here for any invertible sheaf $N$ of $\oo_C$-modules one can define the invertible sheaf $\det R\pi_* ( N)$ of $\oo_S$-modules in the following explicit way by means of  \'etale topology.
Using  the proof of Lemma~\ref{et-n} (see also Proposition~\ref{exact-seq}) for any $s \in S$ we find an \'etale neighbourhood $V$ and a relative effective Cartier divisor $\widetilde{D}$ on $C_V$ such that
$$
(\pi_V)_* \left( N \left( \widetilde{D} \right) \right)    \lrto (\pi_V)_* \left( N \left( \widetilde{D}\right)/N \right)
$$
   is a complex of finite locally  free sheaves of $\oo_S$-modules. Then the determinant of this complex does not depend on the choice of $\widetilde{D}$  up to a canonical isomorphism (see the arguments after formula (4.16)
   in~\cite[Chapter~XIII, \S~4]{ACG}). Therefore various these determinants patch together
canonically into  $\det R\pi_* ( N)$.
\end{nt}

\quash{
{\r We don't use this remark further. Moreover, the second part of this remark also can be omitted, see also~\eqref{compl-AG} above.  Do I have to give the reference to Eiksson and co?
In 3 places it is written approximately the same: in the beginning of proof of item~\ref{general} of Theorem~\ref{Weil-gen}, in the proof of Lemma~\ref{et-n} and in Remark~\ref{Gen-det}. Maybe, to write one Lemma instead of this.}
}

\section{Quintets and the action of the group ind-scheme $\g$}  \label{quint}
We recall the definition of quintet from~\cite[\S~6.1, Def.~1]{O2}.

By {\em a quintet over} $A$ we mean
 a collection $(C, p, \ff, t, e)$ that consists of the following data:
 \begin{itemize}
 \item a
separated morphism  ${C \to \Spec A}$ whose fibres are one-dimensional schemes,
\item
   an $A$-point $p$ of $C$
such that $C$ is smooth in a neighbourhood of $p$ (we will denote by the same letter an $A$-point and its image in $C$),
\item  a  sheaf $\ff$ of $\oo_C$-modules which is  invertible  in a neighbourhood of $p$,
\item  a relative formal parameter $t$ at $p$, i.e.
$t$ is an element of ideal of $p$ of $\hat{\oo}_{C, p}$ that induces an isomorphism of topological $A$-algebras
$
A [[t]]     \xrightarrow{\sim}  \hat{\oo}_{C, p}  $,
\item  a formal trivialization $e$ of $\ff$ at $p$, i.e.  $\hat{\ff}_{C, p}$ is a  free   $\hat{\oo}_{C, p}$-module of rank $1$, and  $e$ is its  basis.
\end{itemize}

We note that $p$ is a relative effective Cartier divisor (see~\cite[\S~6.1]{O2}) and similar to~\eqref{compl-o} we denoted
$$
\hat{\ff}_{C,p} = \varprojlim_{n > 0}  H^0(C, \ff / \ff( - np))   \, \mbox{.}
$$

\medskip

By $\AutL (A)$ we denote {\em  the group of all $A$-automorphisms of the $A$-algebra $A((t))$ that are homeomorphisms}.
There is the following isomorphism of sets (see~\cite[\S~2.2]{O2}, \cite[\S~2.1]{O1} and references therein):
\begin{equation}  \label{isom}
\AutL (A) \,  \simeq  \,   \left\{  \left. h \in L \GG_m(A)   \,  \right|  \,  \nu(h) =1   \right\} \, \mbox{,}
\end{equation}
where the map from the left side to the right side is $\vp  \mapsto \widetilde{\vp}= \vp(t)$, and the map from the right side to the left side is
$
\widetilde{\vp} \longmapsto  \left\{ f \mapsto  f \circ \widetilde{\vp} \right\} \, \mbox{,}
$
where  $f \in A((t))$, and  $f \circ \widetilde{\vp}$ denotes the series from $A((t))$ obtained by substitution of the series $\widetilde{\vp}$ into the series $f$
instead of variable $t$.

From isomorphism~\eqref{isom} we see that $\AutL$ is a  group functor.

We note that the Contou-Carr\`{e}re symbol $\CC$ is invariant under the diagonal action of the group functor $\AutL$ on the commutative group functor
$L\GG_m \times L\GG_m$, see, e.g.,  \cite[\S~2]{OZ}, \cite{GO}.

The group functor  $\AutL$  naturally acts on the group functor $L \GG_M$, therefore the following group functor is well-defined:
$$
{\mathcal G} = L \GG_m \rtimes \AutL  \mbox{,}
$$
where for any elements $h_1, h_2 \in L \GG_m(A)$,  $\vp_1, \vp_2 \in \AutL(A)$ we have
$$
(h_1, \vp_1)(h_2, \vp_2)= (h_1 \vp_1(h_2), \vp_1 \vp_2)  \, \mbox{.}
$$

The group functors $L \GG_m$, $\AutL$ and  $\g$  are represented by  {\em ind-schemes}, see~\cite[\S~2.3]{O2}. The group ind-scheme $\g$ naturally acts on the moduli stack $\mathcal M$ (we consider the Zariski topology), where ${\mathcal M}(A)$ is the groupoid of quintets over $A$, see~\cite[\S~6.1]{O2}. The construction of this action  looks as follows.

 Starting from a quintet $(C, p, \ff, t, e)$  over $A$, i.e. $(C, p, \ff, t, e) \in {\mathcal M}(A)$,  and  from $(h, \vp ) \in \G(A)$ we define  $(C'', p'', \ff'', t'', e'') \in {\mathcal M}(A) $, where
first we define $(C', p', \ff', t', e')  \in {\mathcal M}(A)$ by action of   $\vp \in \AutL(A)$  (see the proofs in~\cite[Theorem~3]{O2}, and also~\cite[\S~17.3]{FB}, \cite[\S~4.1]{Pol}),  then we  define
$(C'', p'', \ff'', t'', e'') \in {\mathcal M}(A)$ by action of $h \in L \GG_m(A)$ on $(C', p', \ff', t', e')$ (see the proofs in~\cite[Theorem~3]{O2}, and also~\cite[\S~18.1.3]{FB}).

By taking smaller  affine open subset in $  \Spec A$
we can suppose that there is an open affine neighbourhood $U$ of $p$ in $C$ such that $U$ is smooth over $\Spec A$,
 the ideal of effective Cartier divisor $p$ restricted to $\Spec \oo_C(U)$  is generated  by a nonzerodivisor of  $\oo_C(U)$,
 and  $\oo_U \simeq \ff \mid_U$ by means a basis $\tilde{e} $ in $\ff(U)$ over $\oo_C(U)$.

 Note that $v= \tilde{e}/e$ is from $L \GG_m(A) \subset \g(A)$, and the action of $v $ on $(C, p, \ff, t, e)$ changes the trivialization from $e$ to $\tilde{e}$. Besides, we have $(h, \vp) = (h \vp(v^{-1}), \vp) (v, 1)$ in the group $\g(A)$. Therefore to construct the action of $(h, \vp)$ on $(C, p, \ff, t, e)$ we will suppose that the trivialization $e$  comes
 from the trivialization of $\ff$ on $U$, i.~e. $e$ is some~$\tilde{e}$.

 We have a  Cartesian diagram
$$
 {\xymatrix{
   {\oo_C(U)} \, \ar[d] \ar@{^{(}->}[r] &  {\oo_C(U \setminus p)} \ar[d]_{\gamma}  \\
   { \hat{\oo}_{C, p}} \, \ar@{^{(}->}[r]   & {{\mathcal K }_{C, p}}
 }}
 $$
 where the bottom arrow is isomorphic to the emebedding $A[[t]]  \hookrightarrow A((t))$.

 Now the curve $C'$ is defined by gluing the new  affine curve $U'$   with $C \setminus p$ along $U \setminus p$, where the ring $\oo_{C'}(U')$ is defined as the fibered product
\begin{equation}  \label{cartesian}
 {\xymatrix{
   {\oo_{C'}(U')} \, \ar[d]^{\alpha} \ar@{^{(}->}[r] &  {\oo_C(U \setminus p)} \ar[d]_{\vp \, \gamma}  \\
   { A[[t]] \, \ar@{^{(}->}[r]}   & {A((t))}
 }}
 \end{equation}
Here $\vp  \, \gamma$ is the composition of the map $\gamma$ and the map $\vp : A((t))  \to A((t))$.
The composition of homomorphisms $\oo_{C'}(U') \to A[[t]] \to A$ defines an $A$-point ${p' }$ of $C'$ such that  $U' \setminus p'  \simeq U \setminus p$ and $C' \setminus p' \simeq C \setminus p$.
The homomorphism $\alpha$ from~\eqref{cartesian} induces the isomorphism of the completion  of $\oo_{C'}(U')$ with respect to the ideal of $p'$ with $A[[t]]$. We define $t' $ as the preimage of $t$ via  this isomorphism.

We define $\ff'$ and $e'$
using the property $\ff' \mid_{C' \setminus p'} = \ff  \mid_{C \setminus p}$ and the isomorphism
 $\ff \mid_U \simeq \oo_U$ by $e$, which induce the isomorphisms  $$\ff' \mid_{U' \setminus p'} \, \simeq  \oo_{U \setminus p}  \simeq \oo_{U' \setminus p'}  \, \mbox{,}$$ and we extend $\ff'$ trivially to $U'$ by means of this composition.  We put $e'=e$.

Using $e'$, we have  the following Cartesian diagram of $\oo(U')$-modules
$$
 {\xymatrix{
   {\ff'(U')} \, \ar[d] \ar@{^{(}->}[r] &  {\ff'(U' \setminus p')} \ar[d]_{\beta}  \\
   {  \hat{\oo}_{C',p'} } \, \ar@{^{(}->}[r]   & {{ {\mathcal K }_{C', p'}  }}
 }}
 $$
Now we define $C''=C'$, $p''=p'$, $t''=t'$  and the sheaf $\ff''$ by gluing the sheaf $\ff'' \mid_{U'} $   with the sheaf $\ff' \mid_{C' \setminus p'}$ over $U' \setminus p'$, where $\ff'' \mid_{U'} $ and the gluing data  are uniquely defined by the fibered product of $\oo(U')$-modules
\begin{equation}  \label{diag-sheaf}
 {\xymatrix{
   {\ff''(U')} \, \ar[d]^{\delta} \ar@{^{(}->}[r] &  {\ff'(U' \setminus p')} \ar[d]_{h \, \beta}  \\
   {\hat{\oo}_{C',p'}   \, \ar@{^{(}->}[r]}   & {{\mathcal K }_{C', p'}}
 }}
 \end{equation}
Here $h \, \beta$ is the composition of the map $\beta$ and  the multiplication by $$h \in  L \GG_m(A) \simeq  A((t'))^* \simeq {\mathcal K }_{C', p'}^*  \, \mbox{.}$$ We define  $e''$  as the preimage of $1$ under the isomorphism which is the completion of the map $\delta$.

The construction of $(C'', p'', \ff'', t'', e'')$ does not depend
on the choice of open $U$ (up to a canonical isomorphism).

An important property is that the action just described does not change the  fibers of a quintet over points of $\Spec A$ and does not change  the topological space of $C$.

\begin{nt}  \label{rem-indepen} \em
From above description  it follows that the action of $L \GG_m(A)  \simeq {\mathcal K }_{C, p}^* $ on $(C, p, \ff, t, e)  \in {\mathcal M}(A)$ does not depend on  $t$.
\end{nt}

\medskip

\begin{prop}  \label{prop-tensor}
Let $q_1 =(C, p, \ff_1, t, e_1) $ and $q_2 = (C, p, \ff_2, t, e_2)$  be from ${\mathcal M}(A)$. Let $q_1 \otimes q_2 = (C, p, \ff_1 \otimes_{\oo_C} \ff_2, t, e_1 \otimes e_2)$. Let    $g, h \in L \GG_m(A)$ and $\vp \in \AutL(A)$. There are canonical isomorphisms, compatible with base change,
$$
g(q_1 \otimes q_2) \simeq g(q_1) \otimes q_2 \simeq q_1 \otimes g(q_2) \, \mbox{,} \quad g h (q_1 \otimes q_2) \simeq g(q_1) \otimes h(q_2) \, \mbox{,} \quad  \vp(q_1 \otimes q_2) \simeq \vp(q_1) \otimes \vp(q_2)  \mbox{.}
$$
\end{prop}
\begin{proof}
The first chain of isomorphisms and the third isomorphism follow from the construction of the action of $\g(A)$ on ${\mathcal M}(A)$ described above. The second isomorphism follows from the first chain of isomorphisms.
\end{proof}

\section{Quintets, the Deligne pairing and the group ind-scheme~$L \GG_m$}
We construct a canonical action of a central extension of the group ind-scheme $L \GG_m$ by $\GG_m$ on a line bundle on the moduli stack of certain quintets.
The central extension is constructed with the help of the Contou-Carr\`ere symbol $\CC$, and the line bundle is constructed with the help of Deligne pairings. We construct also the generalization of this action for  a central extension of the group ind-scheme  $L \GG_m \times L \GG_m$.

\subsection{More about quintets} \label{more_quint}

\begin{defin}   \label{def-refined}
Let $(C, p, \ff, t, e) \in {\mathcal M}(A)$. We say that $(C, p, \ff, t, e)  \in \widetilde{{\mathcal M}}(A)$ if, additionally, the  morphism $\pi : C  \to \Spec A$ is a flat proper finitely presented morphism such that
any fiber of $\pi$ is geometrically reduced and equidimensional of dimension $1$, and  $\ff$ is an invertible sheaf of  $\oo_C$-modules.
\end{defin}

Clearly, we have the moduli stack $\widetilde{\mathcal M}$ (in  the Zariski topology), where $A \mapsto \widetilde{{\mathcal M}}(A)$ gives the groupoid of corresponding quintets over $A$.

\begin{prop}  \label{restr_act}
The natural action of the group ind-scheme $\g$  on the moduli stack $\mathcal M$ gives also the action of $\g$ on $\widetilde{\mathcal M}$.
\end{prop}
\begin{proof}
It follows from the proof of Theorem~3 from~\cite{O2}. Indeed, in the proof of this theorem it was checked that the action of $\g(A)$ on ${\mathcal M}(A) $
preserves the properties of a quintet  $(C, p, \ff, t, e)$ that the natural morphism $C  \to \Spec A$ is flat, finitely presented and proper, and that $\ff$ is an invertible sheaf of $\oo_C$-modules.
Besides, since the action  does not change the  fibers of a quintet, the property that the fibers are geometrically reduced and have one-dimensional irreducible components is preserved under this action.
\end{proof}

\begin{prop}  \label{Prop-rel}
Let $\pi : C \to \Spec A$ be a morphism with conditions as in Definition~\ref{def-refined} for the morphism $\pi$. Let $L$ be an  invertible sheaf of  $\oo_C$-modules. Let $E \subset C$ be a relative effective Cartier divisor. For any point $s \in \Spec A$,  any closed points  $p_1, \ldots , p_r$ on $\pi^{-1}(s)$ such that any $p_i \notin E$,    for any $n \ge 0$ and any $\psi \in H^0(C, L / L(-nE))$
there is an  \'etale  neighbourhood $U$ of  $s$ such that the pullback  of $L$  to $C_{U}$ has a relative meromorphic section $l$
with a divisor $(l)= F_1 -F_2$ such that the relative effective Cartier divisor $F_1 +F_2$ is disjoint from the preimage of $E$ in $C_U$ and does not contain any of the preimages of any~$p_i$,
and the image of $l$  under the natural map  is the pullback of $\psi$ to~$C_U$.
\end{prop}
\begin{proof}
We consider additional closed points $p_{r+1}, \ldots, p_u$ on $\pi^{-1}(s)$ such that any irreducible component of $\pi^{-1}(s)$ contains at least one point $p_i$, where $1 \le i \le u$.

As in the proof of Lemma~\ref{et-n} (see also the proof of item~\ref{ii-last} of Theorem~\ref{Weil-gen}) we will find
an \'etale  affine  neighbourhood $V $ of  $s$ and a relative effective Cartier divisor $D$ on $C_{V}$,
where $\pi_V : C_V \to V$ is the base change,
 such that the divisors $D$ and $\widetilde{E}$ are disjoint, $\widetilde{p_i} \subset  C_V \setminus D$ for every $i$,
 where $\widetilde{E}$ and $\widetilde{p_i}$ are the subschemes in $C_V$ which are   the preimages of  $E$ and $p_i$,
 and   $\oo_{C_{V}}(D)$ is an ample invertible sheaf.
 Let $\widetilde{L}$ be the pullback of $L$ to $C_V$.
 Now there is an integer $v > 0$
 such that for any integer $k > v$ we have $R^q (\pi_V)_* \left(\widetilde{L} \left(kD -n \widetilde{E}\right) \right) = 0$ for any $q >0$ and
  $(\pi_V)_* \left(\widetilde{L} \left(kD - n \widetilde{E} \right) \right)$
   is a finite   free sheaf of $\oo_{V}$-modules,  and these sheaves commute with any base change. Since $V$ is affine,  it implies (see~\cite[Chapter~5, Prop.~2.28]{Liu})
   that $H^1\left(C_V, \,   \widetilde{L} \left(kD -n \widetilde{E} \right) \right) =0$.

   Therefore the following natural homomorphism $\sigma$ is surjective:
   $$
   H^0\left(C_V, \, \widetilde{L} \left(kD \right) \right) \, \stackrel{\sigma}{\longtwoheadrightarrow}  \, H^0\left(C_V, \,  \widetilde{L} (kD ) / \widetilde{L} \left(kD -n \widetilde{E}\right) \right) = H^0\left(C_V, \, \widetilde{L}  / \widetilde{L} \left( -n \widetilde{E} \right) \right)   \, \mbox{.}
   $$
   We fix  an integer $k >v$ such that the following  natural homomorphism $\tau$ is surjective:
   $$
   H^0\left(\widetilde{\pi^{-1}(s)},  \widetilde{L} (kD ) \mid_{\widetilde{\pi^{-1}(s)}} \right)   \stackrel{\tau}{\twoheadrightarrow}   \left(\bigoplus_i H^0\left(\widetilde{p_i},  \widetilde{L} \mid_{\widetilde{p_i}}\right) \right)
   \bigoplus  H^0\left(\widetilde{\pi^{-1}(s)},  \widetilde{L}  / \widetilde{L} \left( -n \widetilde{E} \right) \mid_{\widetilde{\pi^{-1}(s)}} \right)
    \mbox{,}
   $$
where $\widetilde{\pi^{-1}(s)}$  is the preimage of the fiber $\pi^{-1}(s)$ in $C_V$, and the sheaves $\widetilde{L} (kD ) \mid_{\widetilde{\pi^{-1}(s)}}$, $\widetilde{L}  / \widetilde{L} \left( -n \widetilde{E} \right) \mid_{\widetilde{\pi^{-1}(s)}}$ and $\widetilde{L} \mid_{\widetilde{p_i}}$ are the pullbacks of the sheaves $\widetilde{L} (kD )$, $\widetilde{L}  / \widetilde{L} \left( -n \widetilde{E} \right)$
and $\widetilde{L}$ to $\widetilde{\pi^{-1}(s)}$ and $\widetilde{p_i}$ correspondingly.

Let $\widetilde{\psi}$ be the pullback of $\psi$ to $C_V$. Consider an element $\varepsilon \in H^0\left(C_V, \, \widetilde{L} \left(kD \right) \right)$
such that $\sigma(\varepsilon) = \widetilde{\psi}$. Let $\overline{\psi}$ be the image of $\widetilde{\psi}$ in $ H^0\left(\widetilde{\pi^{-1}(s)}, \, \widetilde{L}  / \widetilde{L} \left( -n \widetilde{E} \right) \mid_{\widetilde{\pi^{-1}(s)}} \right)$. Consider an element ${\theta \in H^0\left(\widetilde{\pi^{-1}(s)},  \widetilde{L} (kD ) \mid_{\widetilde{\pi^{-1}(s)}} \right)}$ such that $\tau(\theta) = \left( \oplus_i a_i   \right)  \oplus \overline{\psi} $, where $a_i \ne 0$ for any $i$. Let $\overline{\varepsilon} $  be the image of $\varepsilon$ in $ H^0\left( \widetilde{\pi^{-1}(s)}, \widetilde{L} (kD ) \mid_{\widetilde{\pi^{-1}(s)}}  \right) $.
Then the element $\theta - \overline{\varepsilon}$  belongs to $H^0\left(\widetilde{\pi^{-1}(s)},  \widetilde{L} \left(kD - n\widetilde{E} \right) \mid_{\widetilde{\pi^{-1}(s)}} \right)$, where $  \widetilde{L} \left(kD - n \widetilde{E} \right) \mid_{\widetilde{\pi^{-1}(s)}} $ is the pulback of  $\widetilde{L} \left(kD - n\widetilde{E} \right)$ to $\widetilde{\pi^{-1}(s)}$.

Since $(\pi_V)_* (\widetilde{L} \left(kD - n \widetilde{E} \right) )$
   is a finite   free sheaf of $\oo_{V}$-modules that commute with any base change, there is an element  $\vartheta$  from $H^0 \left(C_V, \,   \widetilde{L} \left(kD -n \widetilde{E} \right) \right)$ that is mapped to
   $\theta - \overline{\varepsilon}$ under the natural map. Then the element $\rho = \vartheta + \varepsilon$ belongs to $H^0\left(C_V, \, \widetilde{L} \left(kD \right) \right)$. And $\rho$ is a relative section
   over $C_U$, where an  open subset $U \subset V $  contains the preimage of $s$ in $V$ and  $\rho$ does not vanish identically on any irreducible component of the fiber over any point of $U$.
   Now $l$ that we need is  $\rho \otimes  1_{kD}^{\otimes -1}$ over $U$, where $1_{kD}$ is the relative section of the pullback of $\oo_{C_{{V}}}(k {D})$  to $C_{U}$ induced by $1 \in H^0 (C_{{U}}, \oo_{C_{{U}}})$.
\end{proof}

\subsection{Deligne pairing and the group ind-scheme $L \GG_m$}  \label{stack}

Let $\pi : C \to \Spec A$ and an $A$-point $p$ of $C$  be with conditions as in Definition~\ref{def-refined}. Let $L_1$ and $L_2$ be two invertible sheaves of $\oo_C$-modules such that
\begin{equation}   \label{cond-incl}
L_2(j p)  \subset L_1  \subset L_2(k p)
\end{equation}
for integers $j \le k$.
\quash{Then there are integers $k' \ge m' $ such that
$$
L(k'p)  \subset N  \subset L(m')  \, \mbox{.}
$$}
Let $N$ be an invertible sheaf of $\oo_C$-modules, and $e$ be a formal trivialization of $N$ at $p$. By this data we will  construct a canonical isomorphism $ {\langle L_1 : L_2 , e  \rangle }$ of invertible sheaves  of $\oo_{\Spec A}$-modules, compatible with  base change,
$$
\langle L_1 : L_2 , e  \rangle   \quad : \quad    \langle L_2 , N  \rangle \, \xrightarrow{\sim}  \,  \langle L_1 , N  \rangle
$$
such that $   \langle L_1 : L_2 ,  e \rangle
 \langle L_2 : L_3 , e \rangle      =  \langle L_1 : L_3 ,  e \rangle
$
for any appropriate $L_1, L_2, L_3$.

From the last property we have that it is enough to construct $ \langle L_1 : L_2 , e  \rangle $ when $L_1 \supset L_2$. Then by~\cite[Lemma~062Y]{Stacks}  we have $L_1 = L_2( R) $, where $R$ is
a relative effective Cartier divisor  on~$C$ such that $C \setminus R = C \setminus p$. Therefore, by Proposition~\ref{prop-Del-pair}  we have an isomorphism
$$
\langle L_2, N   \rangle  \otimes_{\oo_{\Spec A }} \langle  \oo_{C}( R) , N  \rangle    \xrightarrow{\sim}   \langle L_1, N   \rangle    \, \mbox{.}
$$
Now $   \langle L_1 : L_2 ,  e \rangle$ is induced \'{e}tale locally on $\Spec A$ by a non-vanishing section $\langle 1_R, \check{e} \rangle$ of the pullback of $\langle  \oo_{C}( R) , N  \rangle$ over $U$. Here $U$ is an   \'{e}tale neighbourhood of a point in $\Spec A$,  $1_R \in H^0 \left(C_U, \oo_{C_U} \left( \widetilde{R} \right) \right)$  is induced by
$1 \in H^0\left(C_U, \oo_{C_U}\right)$, where $\widetilde{R}$ is the preimage of $R$ to $C_U$,  and $\check{e}$ is a relative meromorphic section of  the pullback $\widetilde{N}$ of $N$ to $C_U$
such that the image of $\check{e}$ in $H^0 \left(C_U, \widetilde{N} / \widetilde{N \left(- ip \right)} \right)$ coincides with the image of $e$, where an integer $i > 0$ satisfies the condition $\oo_C(R) \subset \oo_C(i p)$, and
$\widetilde{N (- ip)}$ is the pullback of  $N (- ip)$ to $C_U$.
Such a section $\check{e}$ exists by Proposition~\ref{Prop-rel}, and $ \langle L_1 : L_2 , e  \rangle $ does not depend on the choice of $i$ and $\check{e}$.

Similarly, we construct a canonical isomorphism $ {\langle e, L_1 : L_2  \rangle }$ of invertible sheaves  of $\oo_{\Spec A}$-modules $
 \langle  N , L_2  \rangle  \xrightarrow{\sim}    \langle  N , L_1  \rangle
$
with analogous properties and which is compatible with  $ \langle L_1 : L_2 , e  \rangle $ and  the isomorphism in item~\ref{iii-3} of Proposition~\ref{prop-Del-pair}.

\medskip
For any quintet $q = (C, p, \ff, t, e) \in {\mathcal M}(A)$ and any element $g \in L \GG_m (A) \subset \g(A)$ we {\em denote} the quintet $g(q) = (g(C), g(p), g(\ff), g(t), g(e)) $.

If $g \in L\GG_m(A)$, then the sheaves $\ff$ and $g(\ff)$ satisfy the condition~\eqref{cond-incl}. This follows from the construction of the action of $g$ and since there are  integers $r \ge w$ such that
$$
 t^r \hat{\oo}_{C,p} \,  \subset  \, g \hat{\oo}_{C,p} \,  \subset \, t^w \hat{\oo}_{C,p}  \, \mbox{.}
$$

Let $(C, p, L, e_L)$ and $(C, p, M, e_M)$ be parts of quintets from $\widetilde{{\mathcal M}}(A)$ (here $L$ and $M$ are invertible sheaves of $\oo_C$-modules, $e_L$ and $e_M$ be  formal trivializations of $L$ and $M$ at $p$ correspondingly, and we do not fix an existing relative formal parameter $t$ at $p$ for both quintets). For any  $g_1$ and $g_2$ from $ {\mathcal K }_{C, p}^* \simeq L \GG_m(A)   $  we {\em define}
an isomorphism $T_{\langle L, M  \rangle} (g_1, g_2)$
between invertible sheaves  of $\oo_{\Spec A}$-modules as the following composition (recall also Remark~\ref{rem-indepen}):
$$
T_{\langle L, M \rangle}(g_1, g_2)
\quad :  \quad
\langle L, M  \rangle  \xrightarrow{ \langle g_1(L) : L   , \, e_M\rangle} \langle g_1(L), M  \rangle   \xrightarrow{\langle g_1(e_{L}) , \, g_2(M) : M  \rangle } \langle g_1(L), g_2(M) \rangle  \,
 \mbox{.}
$$

\begin{Th}  \label{Th-CC}
Let $(C, p, L, e_L)$ and $(C, p, M, e_M)$ be parts of quintets from $\widetilde{{\mathcal M}}(A)$ (where we do not fix  an existing relative formal parameter $t$ at $p$ for both quintets), and $g_1, g_2, h_1, h_2$  be from ${\mathcal K }_{C, p}^* \simeq L \GG_m(A)   $. We have
$$
T_{\langle g_1(L), \, g_2(M)   \rangle} (h_1, h_2)  \, T_{\langle L, M \rangle}(g_1, g_2) = \CC(h_1,g_2) \, T_{\langle L,  M \rangle}(h_1g_1, h_2 g_2)   \, \mbox{.}
$$
\end{Th}
\begin{proof}
We denote an element $P(h_1, h_2, g_1, g_2)$ from $ A^*$ such that
$$
T_{\langle g_1(L), \, g_2(M)   \rangle} (h_1, h_2)  \, T_{\langle L, M \rangle}(g_1, g_2) = P(h_1, h_2, g_1,g_2) \, T_{\langle L,  M \rangle}(h_1g_1, h_2 g_2)    \, \mbox{.}
$$
From the properties
\begin{gather*}
 \langle  h_1g_1(L) : g_1(L) , \, e_M \rangle \,  \langle  g_1(L) : L  , \, e_M  \rangle = \langle h_1g_1(L) : L , \, e_M   \rangle \, \mbox{,}\\
 \langle h_1g_1(e_L) , \,  h_2g_2(M) : g_2(M)  \rangle \, \langle  h_1g_1(e_L), \, g_2(M) : M \rangle  = \langle  h_1g_1(e_L) , \, h_2g_2(M) : M \rangle
\end{gather*}
it is easy to see
that $P(h_1, h_2,g_1, g_2)$ is the quotient of the composition of left and lower isomorphisms and the composition of upper and right isomorphisms in the following (in general, non-commutative) diagram:
\begin{equation}   \label{diagr}
 {\xymatrix{
   {\langle g_1(L)   , M  \rangle } \ar[d]_{\langle g_1(e_L) , \, g_2(M) : M  \rangle}  \ar[rrrr]^{\langle h_1g_1(L) : g_1(L), \, e_M  \rangle} &&&&  {\langle h_1g_1(L), M   \rangle}  \ar[d]^{\langle h_1g_1(e_L)  , \, g_2(M) : M  \rangle}  \\
   {\langle  g_1(L), g_2(M)   \rangle} \ar[rrrr]_{\langle h_1g_1(L) : g_1(L)  \, , \, g_2(e_M)    \rangle}   &&&& {\langle h_1g_1(L) ,  g_2(M)  \rangle}
 }}
 \end{equation}

We fix now a relative formal parameter $t$ at $p$. Thus, quintets $(C, p, L, t, e_L)$ and $(C, p, M, t,  e_M)$ are from $\widetilde{{\mathcal M}}(A)$.

The expression $P(h_1, h_2, g_1, g_2)$ is functorial
with respect to the ring $A$, i.e. with respect to ring homomorphisms $A \to A'$. Hence the equality ${P(h_1, h_2, g_1, g_2) = \CC(h_1,g_2)}$ is enough to prove \'{e}tale locally on $\Spec A$.
Besides, without loss of generality, changing $S = \Spec A$ by its \'{e}tale cover,  we assume that $t$ comes from ${\mathcal R}_{C/S}^*(\oo_C)$. (This can be done by Proposition~\ref{Prop-rel} \'{e}tale locally on $\Spec A$.)

Let the integers $\nu(h_1)_1, \ldots , \nu(h_1)_n$  and $\nu(g_2)_1, \ldots , \nu(g_2)_m$
be the values of
the locally constant $\Z$-valued functions $\nu(h_1) $ and $\nu(g_2)$ on $\Spec A$.
Fix an integer $k$ which is greater than any of above integers.

Let $h_1 = \sum_i a_i t^i$  and $g_2 = \sum_i b_i t^i$, where any $a_i$ and $b_i$ are from $A$. Then from Section~\ref{sec-loop} it follows that for any elements $x_i$ and $y_i$ from a commutative $A$-algebra $A'$, where $i > k$, the  elements
$$
\tilde{h}_1 = \sum_{i \le k} a_i t^i + \sum_{i > k} x_i t^i                                 \quad  \mbox{and}  \quad   \tilde{g}_2= \sum_{i \le k} b_i t^i + \sum_{i > k} y_i t^i
$$
 belongs to $A'((t))^*$, where we denoted by the same letters   $a_i$ and $b_i$  their images in $A'$.
The maps on $A'$-points (for various $A'$)
\begin{equation}  \label{morph}
\{x_i \}_{i > k} \, \mbox{,}   \, \{ y_i  \}_{i > k} \, \longmapsto \, P\left(\tilde{h}_1, h_2, g_1, \tilde{g}_2\right) \in A'^{*}
\end{equation}
defines the morphism of schemes
\begin{equation}   \label{sch-mor}
\Spec A[\{X_i\}_{i >k}, \{Y_i\}_{i > k}] \lrto \Spec A[v,v^{-1}]   \, \mbox{}
\end{equation}
where instead of variables $X_i$ and $Y_i$ we insert elements $x_i$ and $y_i$ from $A'^*$. (Here  we denoted by the same letters $h_2$ and $g_1$ their images  in $A'((t))$.) The  morphism~\eqref{sch-mor} is uniquely defined
by a polynomial from $A[\{X_i\}_{i >k}, \{Y_i\}_{i > k}]$ which is the image of $v$ under the corresponding ring homomorphism.
Hence there is an integer $j_1 > k$ such that for any $\hat{h}_1$ and $\hat{g}_2$ from $A'((t))$ with the property that $\hat{h}_1 - h_1 \in t^{j_1} A'[[t]]$
and $\hat{g}_2 - g_2 \in t^{j_1} A'[[t]]$ we have that $\hat{h}_1$ and $\hat{g}_2$ belong to $A'((t))^*$ and
$$
P(h_1, h_2, g_1, g_2) = P\left(\hat{h}_1, h_2, g_1 , \hat{g}_2\right) \, \mbox{.}
$$

 Analogous property is satisfied for the morphism given as $\left(\tilde{h}_1, \tilde{g}_2 \right) \mapsto \CC\left(\tilde{h}_1, \tilde{g}_2\right)$ instead of $P$ in~\eqref{morph} for an integer $j_2$ instead of $j_1$. Let $j = \max (j_1, j_2)$.

By Proposition~\ref{Prop-rel} applied to $\oo_C(d p)$ for sufficiently large integer $d$, for any  ${s \in \Spec A}$ there is an  \'{e}tale neighbourhood $V =\Spec A''$ such that  there are two relative meromorphic  functions from ${\mathcal R}_{C_V/V}^*\left(\oo_{C_V}\right)$ whose images  coincide with the images of $h_1$ and $g_2$    in the quotient  group $A''((t)) / t^j A''[[t]]$,
and  the divisors of these  meromorphic functions are disjoint on $(C \setminus p)_V$. (We used that ${{\mathcal R}_{C_V/V}^*\left(\oo_{C_V}(d p_V) \right) = {\mathcal R}_{C_V/V}^*\left(\oo_{C_V}\right)}$.)

From diagram~\eqref{diagr} we have that $P(h_1, h_2, g_1, g_2)$  does not depend on the formal trivializations $e_L$ and $e_M$. Indeed,   by Proposition~\ref{prop-tensor}, $g_2(M) \simeq g_2(\oo_C) \otimes_{\oo_C} M$ and $ g_2(e_M) = g_2(e_{\oo_C}) \otimes e_M$,
where the trivialization $e_{\oo_C}$ comes from $1 \in H^0(C, \oo_C)$, and therefore the upper isomorphism in~\eqref{diagr} will change by multiplication with the same element from $A^*$ as the lower isomorphism  when we change $e_M$ by another formal trivialization. Analogously we obtain with the left and right isomorphisms in~\eqref{diagr}
when we change $e_L$. Hence, to compute $P(h_1, h_2, g_1, g_2)$ we can suppose that formal trivializations $e_L$ and $e_M$ \'{e}tale locally on $\Spec A$  come from relative meromorphic sections of $L$ and $M$  that exist
by Proposition~\ref{Prop-rel}.

Now, without loss of generality, changing $S = \Spec A$ by its \'{e}tale cover,  to prove that $P(h_1,h_2, g_1, g_2) = \CC (h_1,g_2)$ we assume that $h_1$ and $g_2$ come from ${\mathcal R}_{C/S}^*(\oo_C)$, the divisors $(h_1)$ and $(g_2)$ are disjoint on $C \setminus p$, and the formal trivializations  $e_L$ and $e_M$
come from ${\mathcal R}_{C/S}^*(L)$ and ${\mathcal R}_{C/S}^*(M)$ correspondingly. (It follows from the discussion above that this can be done   \'{e}tale locally on $\Spec A$.)
Besides, from the construction it follows that the formal trivializations
$g_2(e_{\oo_C})$,
$g_2(e_M) = g_2(e_{\oo_C}) \otimes e_M$, and analogously $g_1(e_{\oo_C})$, $g_1(e_L) $, $h_1(e_{\oo_C})$ and  $h_1g_1(e_L)$ come from relative meromorphic sections of the corresponding invertible sheaves, and
\begin{gather}  \label{triv}
g_2(e_{\oo_C}) \mid_{C \setminus p} = g_2^{-1}  \mid_{C \setminus p}  \mbox{,} \quad
h_1(e_{\oo_C}) \mid_{C \setminus p} = h_1^{-1}  \mid_{C \setminus p}  \mbox{,}  \\ \nonumber g_2(\oo_C) \mid_{C \setminus p} = h_1(\oo_C)  \mid_{C \setminus p} = \oo_{C \setminus p} \mbox{.}
\end{gather}
We will denote these relative meromorphic sections by the same letters.

By Proposition~\ref{prop-tensor},  $h_1g_1(L)  \simeq h_1(\oo_C) \otimes_{\oo_C} g_1(L)$. Hence, from the construction and item~\ref{itt-1} of Proposition~\ref{prop-Del-pair}, we have that the isomorphism $\langle h_1g_1(L) : g_1(L) , \, e_M \rangle$
in diagram~\eqref{diagr}
is induced by the tensor product with the section $\langle 1_{h_1} , e_M  \rangle $  of  $\langle h_1(\oo_C),  M    \rangle  $, where $1_{h_1}$
is the unique relative meromorphic section of $h_1(\oo_C)$ such that $1_{h_1} \mid_{C \setminus p} = 1$, because $h_1(\oo_C) \mid_{C \setminus p} = \oo_{C \setminus p}$
and $h_1(\oo_C)= \oo_C(-wp + R)$, where an integer $w > 0$ and $R$ is a relative effective Cartier divisor on $C$. Analogously, the isomorphism
$\langle h_1g_1(L) : g_1(L) , \, g_2(e_M) \rangle$
 is induced by the tensor product with the section $\langle  1_{h_1}, g_2(e_M) \rangle $  of  $\langle h_1(\oo_C) , g_2(M)     \rangle  $,
the isomorphism $\langle g_1(e_L) , \, g_2(M) : M   \rangle$ is induced by the tensor product with the section $\langle g_1(e_L) , 1_{g_2}  \rangle $  of  $\langle g_1(L)   ,  g_2(\oo_C) \rangle  $,
and the isomorphism
$\langle h_1g_1(e_L) ,  g_2(M) : M   \rangle$
 is induced by the tensor product with the section
 $\langle h_1g_1(e_L) , 1_{g_2}  \rangle $
   of
   $\langle h_1g_1(L)   ,  g_2(\oo_C) \rangle  $.

Besides, by Proposition~\ref{prop-tensor}, we have
\begin{gather*} \langle 1_{h_1}, g_2(e_M)   \rangle = \langle 1_{h_1} , g_2(e_{\oo_C})   \rangle   \otimes \langle 1_{h_1} , e_M   \rangle  \,  \mbox{,}
   \\
   \langle h_1g_1(e_L) , 1_{g_2}  \rangle = \langle h_1(e_{\oo_C}) , 1_{g_2}  \rangle   \otimes  \langle g_1(e_L) , 1_{g_2} \rangle
    \mbox{.}
    \end{gather*}

Hence and from the diagram~\eqref{diagr} it follows that $P(h_1, h_2, g_1, g_2)$ coincides with the quotient of sections  $\langle 1_{h_1} , g_2(e_{\oo_C})    \rangle$ and   $\langle h_1(e_{\oo_C}) , 1_{g_2}   \rangle$ in $\langle h_1(\oo_C) ,
g_2(\oo_C) \rangle$. From the construction of $\langle h_1(\oo_C) ,
g_2(\oo_C) \rangle$ and formulas~\eqref{triv} it follows that
\begin{gather}   \label{tr1}
\langle  h_1 \left(e_{\oo_C}\right) , g_2(e_{\oo_C})     \rangle  =  h_1^{-1}   ( g_2^{-1}  \mid_{C \setminus p} )   \langle 1_{h_1} , g_2(e_{\oo_C})  \rangle
=  h_1   ( g_2  \mid_{C \setminus p} )   \langle 1_{h_1} , g_2(e_{\oo_C}) \, \mbox{,}\\
\label{tr2}
\langle  h_1 \left(e_{\oo_C} \right) , g_2(e_{\oo_C})     \rangle  =  g_2^{-1}   ( h_1^{-1}  \mid_{C \setminus p} )   \langle h_1(e_{\oo_C}) , 1_{g_2}   \rangle
= g_2   ( h_1  \mid_{C \setminus p} )   \langle h_1(e_{\oo_C}) , 1_{g_2}   \rangle
 \mbox{.}
\end{gather}
Hence and from the reciprocity law proved in items~\ref{ii-4} and~\ref{ii-last} of Theorem~\ref{Weil-gen}  we have
$$
P(h_1, h_2, g_1, g_2) =  g_2   ( h_1  \mid_{C \setminus p} ) \, h_1   ( g_2  \mid_{C \setminus p} ) ^{-1} = \CC(h_1,g_2)   \, \mbox{.}
$$
\end{proof}
\begin{nt}  \em
For any  $g_1$ and $g_2$ from ${\mathcal K }_{C, p}^* \simeq L \GG_m(A)   $  we can define
another  isomorphism between invertible sheaves  of $\oo_{\Spec A}$-modules $\langle L, M  \rangle$ and $\langle g_1(L), g_2(M)  \rangle$:
$$
\widetilde{T}_{\langle L, M \rangle}(g_1, g_2)  \quad :  \quad
\langle L, M   \rangle  \xrightarrow{\langle e_L , \, g_2(M) : M \rangle }   \langle L , \, g_2(M)  \rangle   \xrightarrow{\langle g_1(L) : L , \, g_2(e_M) \rangle}  \langle g_1(L), g_2(M)  \rangle \, \mbox{.}
$$
We will have $\widetilde{T}_{\langle L, M \rangle}(g_1, g_2)  = \CC(g_1, g_2) \, T_{\langle L, M \rangle} (g_1, g_2) $,
and then  Theorem~\ref{Th-CC}  implies
$$   \widetilde{T}_{\langle g_1(L), g_2(M)   \rangle} (h_1, h_2)  \, \widetilde{T}_{\langle L, M \rangle}(g_1, g_2) = \CC(h_2,g_1) \, \widetilde{T}_{\langle L,  M \rangle}(h_1g_1, h_2 g_2)    \mbox{.} $$
Indeed, to prove the first equality, reasoning as in  the proof of Theorem~\ref{Th-CC} and changing $\Spec A$ by its \'{e}tale cover, it is enough to compute   that $$\widetilde{T}_{\langle L, M \rangle}(g_1, g_2) \, {T}_{\langle L, M \rangle}(g_1, g_2)^{-1} =
\langle 1_{g_1} ,  g_2(e_{\oo_C})  \rangle
\langle g_1(e_{\oo_C}) , 1_{g_2}  \rangle^{-1}
$$
when $g_1$ and $g_2$ come from ${\mathcal R}_{C/S}^*(\oo_C)$ and their divisors are disjoint on $C \setminus p$.
Now similarly to \eqref{tr1}-\eqref{tr2} and using  items~\ref{ii-4} and~\ref{ii-last} of Theorem~\ref{Weil-gen} we have
$$
\langle 1_{g_1} ,  g_2(e_{\oo_C})  \rangle
\langle g_1(e_{\oo_C}) , 1_{g_2}  \rangle^{-1} = \CC(g_1, g_2)  \, \mbox{.}
$$
\end{nt}

\medskip

By a  quintet $(C,p, \mathcal F,t,e) \in \widetilde{\mathcal M}(A)$ we  construct  an invertible sheaf $\langle \mathcal F, \mathcal F \rangle$ of $\oo_{\Spec A}$-modules. This defines {\em a line
bundle} $\langle \mathfrak{F}, \mathfrak{F}   \rangle$ {\em on  the moduli stack} $\widetilde{\mathcal M}$.
Since $L \GG_m  \subset \g$,  from Proposition~\ref{restr_act} we have the natural action of the group ind-scheme $L \GG_m$ on the moduli stack $\widetilde{\mathcal M}$.
Theorem~\ref{Th-CC} immediately implies the following theorem.
\begin{Th} \label{stack-th}
Consider
the group ind-scheme which is the central extension of $L \GG_m$ by $\GG_m$ given  by the $2$-cocycle $\CC$ (see also~\cite[\S~3.1]{O2} and~\cite[Prop.~2.2]{O1}).
 The isomorphisms $ g \mapsto T_{\langle {\mathcal F}, {\mathcal F}   \rangle} (g,g)$ define the natural  action of this group ind-scheme  on the line bundle $\langle \mathfrak{F}, \mathfrak{F}   \rangle$  on the moduli stack $\widetilde{\mathcal M}$.
\end{Th}

We have the $2$-cocycle on the group ind-scheme ${L\GG_m \times L \GG_m}$ with values in $\GG_m$:
$$
(h_1 \times h_2, \,  g_2 \times g_2)   \longmapsto \CC(h_1, g_2)  \, \mbox{,}
$$
which defines the group ind-scheme $\mathcal H$ that is the central extension of ${L\GG_m \times L \GG_m}$ by~$\GG_m$.

Consider the moduli stack $\widetilde{\mathcal M}^{\langle 2\rangle}$ (in Zariski topology) that associates with every $A$ the groupoid $\widetilde{\mathcal M}^{\langle 2\rangle}(A)$ of data $(C,p, {\mathcal F}_1, {\mathcal F}_2, t,e_1, e_2 )$,
 where $(C,p, {\mathcal F}_i,t,e_i) \in \widetilde{\mathcal M}(A)$ for ${i = 1,2}$. The action of $L \GG_m$ on $\widetilde{\mathcal M}$ gives the natural action of ${L\GG_m \times L \GG_m}$ on $\widetilde{\mathcal M}^{\langle 2\rangle}$.
 On $\widetilde{\mathcal M}^{\langle 2\rangle}$ consider the line bundle $\langle {\mathfrak F}_1, {\mathfrak F}_2  \rangle $, where  $\langle {\mathfrak F}_1, {\mathfrak F}_2  \rangle $ is defined by
$\langle {\mathcal F}_1, {\mathcal F}_2 \rangle$ on $\Spec A$. Theorem~\ref{Th-CC} immediately implies the following theorem.
 \begin{Th} \label{stack-th-2}
The  isomorphisms $g_1 \times g_2 \mapsto T_{\langle {\mathcal F}_1, {\mathcal F}_2   \rangle} (g_1,g_2)$ define the natural action of
 the group ind-scheme $\mathcal H$ on the line bundle $\langle {\mathfrak F}_1, {\mathfrak F}_2  \rangle $
on the moduli stack $\widetilde{\mathcal M}^{\langle 2\rangle}$.
\end{Th}

\section{Quintets, the Deligne pairing and  the group ind-scheme~$\g$}
We construct canonical actions of  central extensions of the group ind-scheme $\g$ by $\GG_m$ on  line bundles on  moduli stacks of certain quintets.
 These line bundles are constructed with the help of Deligne pairings.

\subsection{Deligne pairing and the group ind-scheme $ \AutL$} \label{del-aut}

Let $(C, p, L, t, e_L)$ and $(C, p, M, t,  e_M)$ be quintets from $\widetilde{{\mathcal M}}(A)$.
Let $\vp$ be from  $ \AutL (A)$. We will  construct a canonical isomorphism $ T_{\langle L , M \rangle }(\vp)$ of invertible sheaves  of $\oo_{\Spec A}$-modules,
compatible with base change,
$$
 T_{\langle L , M \rangle }(\vp)  \quad : \quad    \langle L , M \rangle \, \xrightarrow{\sim}  \,  \langle \vp(L) , \vp(M)  \rangle  \, \mbox{,}
$$
where $\vp(L)$ and $\vp(M)$ are invertible sheaves from the quintets $\vp((C, p, L, t, e_L))$ and $\vp ((C, p, M, t,  e_M))$ correspondingly.

Let an integer $k > 0$ such that $\vp (t^k A[[t]])  \subset tA[[t]]$.
By Proposition~\ref{Prop-rel}, for any $s \in \Spec A$ there is an \'{e}tale neighbourhood $U$ of $s$ and relative meromorphic sections $l$ and $m$ of the pullbacks
$\widetilde{L}$ of $L$ and $\widetilde{M}$ of $M$ to $C_U$ correspondingly such that  the images of $l$ and $m$ in $H^0\left(C_U, \widetilde{L}/ \widetilde{L(-kp)} \right)$ and $H^0\left(C_U, \widetilde{M}/ \widetilde{M(-kp)} \right)$ coincide with the images of $e_L$ and $e_m$ correspondingly
and the divisors $(l)$ and $(m)$ are disjoint. Here $\widetilde{L(-kp)}$ and $\widetilde{M(-kp)}$
 are the pullbacks of  $L(-kp)$ and $M(-kp)$ to $C_U$.
Since  it is enough to construct $T_{\langle L , M \rangle }(\vp)$ \'{e}tale locally on $\Spec A$, without loss of generality, by changing $\Spec A$ to its \'{e}tale cover, we suppose that $l$ and $m$
with above properties exist on $C$ over $\Spec A$.

We consider the quintets $q=(C, p, L, t, l)$ and
$\vp(q)$. Since ${\vp(l / e_l) \in A[[t]]^*}$, and ${\vp \cdot (e_l/l) = \vp( e_l/l) \cdot \vp}$ in the group $\g$, the invertible sheaf  in the quintet $\vp(q)$
is isomorphic to the invertible sheaf in the quintet $\vp((C, p, L, t, e_L)) = (\vp(C), \vp(p), \vp(L), \tilde{t}, \tilde{e}_L)$. (We used that the action of an element from $A[[t]]^*$ on a quintet maps  the sheaf from the quintet to the isomorphic sheaf and change the formal trivialization, see~diagram~\eqref{diag-sheaf}.) We define now the relative meromorphic section $\vp(l)$
of $\vp(L)$ by properties that $\vp(l) \mid_{C \setminus p} = l \mid_{C \setminus p} $ and the image of $\vp(l)$  in $\widehat{\vp(L)}_{\vp(C), \vp(p)}$ coincides with
the formal trivialization from the quintet $\vp(q)$  (recall the construction of $\vp(q)$ from Section~\ref{quint}).
We used here  that ${C \setminus p = \vp(C) \setminus \vp(p)}$ and ${L \mid_{C \setminus p} = \vp(L)  \mid_{C \setminus p}}$.
Similarly we define   the relative meromorphic section $\vp(m)$
of $\vp(M)$ such that  $\vp(m) \mid_{C \setminus p} = m \mid_{C \setminus p} $. We define now $T_{\langle L , M \rangle }(\vp) $ as
$$
T_{\left\langle L , M \right\rangle }(\vp)  \quad : \quad \langle l , m \rangle \longmapsto \langle  \vp(l) , \vp(m) \rangle  \, \mbox{.}
$$
This definition does not depend on the choice of corresponding $l$ and $m$, since ${\vp(l) \mid_{C \setminus p} = l \mid_{C \setminus p}}$,
${\vp(m) \mid_{C \setminus p} = m \mid_{C \setminus p} }$, and $l, m, \vp(l), \vp(m)$  give the trivializations of $L$, $M$, $\vp(L)$, $\vp(M)$
in open neighbourhoods of $p$ and $\vp(p)$ on $C$ and $\vp(C)$ correspondingly.

\begin{prop}  \label{aut-comp} Let $(C, p, L, t, e_L)$ and $(C, p, M, t,  e_M)$ be  from $\widetilde{{\mathcal M}}(A)$,  $\vp, \phi$ be  from $ \AutL (A)$, and  $g_1, g_2$ be from $L \GG_m(A)$.  We have properties:
\begin{gather*}
 T_{\langle \phi(L) , \, \phi(M) \rangle }(\vp)  \,  T_{\langle L , \, M \rangle }(\phi) = T_{\langle L , \, M \rangle }(\vp \, \phi) \, \mbox{,}
 \\
  T_{\langle g_1(L) , \, g_2(M) \rangle }(\vp) \,  T_{\langle L , \, M \rangle }(g_1, g_2) =  T_{\langle \vp(L) , \, \vp(M) \rangle }(\vp(g_1), \vp(g_2))    \,   T_{\langle L , \, M \rangle }(\vp)  \, \mbox{.}
\end{gather*}
\end{prop}
\begin{proof}
The first formula follows directly from the definition. The second formula follows from the commutativity of the diagrams:
\begin{equation}   \label{diagr2}
 {\xymatrix{
   {\langle L, M  \rangle } \ar[d]_{T_{\langle L , \, M   \rangle} (\vp)}  \ar[rrrrr]^{\langle g_1(L) : L   , \, e_M  \rangle} &&&&&  {\langle g_1(L), M \rangle}  \ar[d]^{T_{\langle g_1(L) , \, M   \rangle} (\vp)}  \\
   {\langle  \vp(L), \vp(M)   \rangle} \ar[rrrrr]_{\left\langle \vp(g_1)(\vp(L)) : \vp(L)   , \, \vp(e_M )  \right\rangle}   &&&&& {\langle \vp(g_1)(\vp(L)) ,  \vp(M)  \rangle}
 }}
 \end{equation}
and similar (symmetric, see item~\ref{iii-3} of Proposition~\ref{prop-Del-pair}) diagram with ${\langle g_1(e_{L}) , \, g_2(M) : M  \rangle}$ instead of ${\langle g_1(L) : L   , \, e_M  \rangle}$ and so on for other maps.

\end{proof}

\subsection{Deligne pairing and the group ind-scheme $\g$}  \label{last}

Let $(C, p, t)$ be a part of a quintet from ${\mathcal M}(A)$. For example, $\left(C, p, t, \oo_C, e_{\oo_C} \right)$ is from ${\mathcal M}(A)$,
where
$e_{\oo_C}$ comes from $1 \in H^0(C, \oo_C)$.
Let $\Omega_{C/A}$ be the sheaf of relative differential $1$-forms on $C$ over $\Spec A$. The sheaf $\Omega_{C/A}$ is a quasi-coherent sheaf of  $\oo_C$-modules, compatible with any base change and
 invertible on the smooth locus of the morphism $C \to \Spec A$,  see~\cite[Section~01UM, Lemma 02G1]{Stacks}. Therefore $ \left(C, p, t, \Omega_{C/A}, dt \right)$ is from ${\mathcal M}(A)$.

\quash{
Let $\pi : C \to \Spec A$ be a morphism with properties required for  the morphism from a quintet from $\widetilde{{\mathcal M}}(A)$. Since every fiber of $\pi$ is Cohen-Macaulay, there is the relative dualizing sheaf $\omega_{C/A}$ of $\pi$ on $C$, which is a pseudo-coherent sheaf of $\oo_C$-modules and  it is flat over $ A$, and  compatible with any base change, see~\cite[Lemma~0E6P]{Stacks}. The sheaf  $\omega_{C/A}$ is defined up to a  canonical isomorphism. For every $\pi$ as above we fix the sheaf $\omega_{C/A}$.
When the morphism $\pi$ is locally projective,
  the sheaf $\omega_{C/A}$ is canonically  isomorphic to the sheaf $\Omega^1_{C/A}$ on the smooth locus of $\pi$, see~\cite[Prop.~(22)]{K}. From the proof of~\cite[Prop.~(22)]{K} it follows that this isomorphism is compatible with base change by any open subset.
When every fiber of $\pi$ is Gorenstein,  the sheaf $\omega_{C/A}$ is invertible, see~\cite[Lemma~0E6R]{Stacks}.}

{\em We will say that a quintet} $(C, p, \ff, t, e) \in \widetilde{{\mathcal M}}(A)$ belongs to $\widetilde{{\mathcal M}}_{\rm sm}(A)$ if
the morphism $C \to \Spec A$ is smooth. We have the moduli stack (in Zariski topology) $\widetilde{{\mathcal M}}_{\rm sm}$ that associates with every  $A$ the groupoid  $\widetilde{{\mathcal M}}_{\rm sm}(A)$.
Since  the action of $\g(A)$ does not change $C \setminus p$,  the natural action of $\g$ on $\widetilde{{\mathcal M}}$
gives also the action of $\g$ on $\widetilde{{\mathcal M}}_{\rm sm}$.

Suppose that $\left(C, p, \ff, t, e \right)$ belongs to   $\widetilde{{\mathcal M}}_{\rm sm}(A)$.
Then $\left(C, p, \Omega_{C/A}, t, dt \right)$ belongs to $\widetilde{{\mathcal M}}_{\rm sm}(A)$. And  quintets of the last kind can be glued to objects over non-affine base schemes in $\widetilde{{\mathcal M}}_{\rm sm}$.

On the moduli stack $\widetilde{{\mathcal M}}_{\rm sm}$ we have the line bundles $\langle  \mathfrak{F}, \boldsymbol{\Omega} \rangle$
and $\langle \bOmega, \bOmega \rangle$ which are defined  by association with every quintet $(C, p, \ff, t, e)  \in \widetilde{{\mathcal M}}_{\rm sm}(A)$  the invertible sheaves
$\langle \ff, \Omega_{C/A} \rangle$  or $\langle \Omega_{C/A} , \Omega_{C/A} \rangle$ on $\Spec A$ correspondingly.

Analogously, we have the line bundle $\langle \bOmega, \bOmega\rangle$ on the moduli stack (in Zariski topology) $\widetilde{{\mathcal M}}_{\rm curve}$,
where $\widetilde{{\mathcal M}}_{\rm curve}$
  associates with every $A$ the groupoid  $\widetilde{{\mathcal M}}_{\rm curve}(A)$ of data
$ (C, p, \oo_C, t, e_{\oo_C})$ such that
$(C, p, \oo_C, t, e_{\oo_C})$
belongs to $\widetilde{{\mathcal M}}_{\rm sm}(A)$.

Besides, before Theorem~\ref{stack-th} we defined the line bundle $\langle  \mathfrak{F}, \mathfrak{F} \rangle$
on the moduli stack~$\widetilde{{\mathcal M}}$.

\medskip

Through the natural morphism of group ind-schemes  $\G \to \AutL  $ we obtain that the group ind-scheme $\G$ naturally acts on the commutative group ind-scheme $L \GG_m$.

Let $\lambda_1$ and $\lambda_2$ be any $1$-cocycles on $\G$ with coefficients in $L \GG_m$. This means that $\lambda_i$ is a morphism of group functors $\G \to L \GG_m$  with the property
$$\lambda_i (g_1, g_2) = \lambda_i(g_1) g_1(\lambda_i(g_2))$$
 for any $g_1, g_2$ from $\g(A)$ and any $i$, see more on cohomology of group functors in~\cite[\S~3]{O2}.
Consider the $2$-cocycle $\lambda_1 \cup \lambda_2$, which is the morphism of group functors  ${\g \times \g   \to L \GG_m \otimes L \GG_m}$ (where
$\left( L \GG_m \otimes L \GG_m \right) (A) = L \GG_m(A)\otimes_{\dz} L \GG_m (A)$) defined by the rule $$(\lambda_1 \cup \lambda_2) (g_1, g_2) = \lambda_1(g_1)  \otimes g_1 (\lambda_2(g_2)) \, \mbox{.} $$
 We construct {\em the $2$-cocycle}
$\langle \lambda_1 , \lambda_2   \rangle$ on $\G$ with coefficients in  $\GG_m$, where $\G$ acts trivially on $\GG_m$,
$$
\langle \lambda_1 , \lambda_2   \rangle  = \CC \circ (\lambda_1 \cup \lambda_2)  \, \mbox{,}
$$
where $\circ$ means the composition of the morphism $\lambda_1  \cup  \lambda_2$ from $\G \times \G$  to $L \GG_m \otimes L \GG_m  $ and
 the $\G$-equivariant morphism  $\CC$ (the Contou-Carr\`{e}re symbol) from $ L \GG_m \otimes  L \GG_m $ to  $\GG_m $.

 The $2$-cocycle $\langle \lambda_1 , \lambda_2   \rangle$ naturally defines the group ind-scheme that is a central extension of $\g$ by $\GG_m$, see~\cite[Prop.~2.2]{O1}.
Explicitly, this central extension as the ind-scheme (not as the  group ind-scheme) is $\GG_m \times \g $, and the group law is given as
\begin{equation}  \label{coc-id}
(a_1, g_1)(a_2, g_2) = (a_1 a_2 \, \langle \lambda_1 , \lambda_2   \rangle (g_1, g_2), \, g_1 g_2) \, \mbox{,}
\end{equation}
where $a_i \in \GG_m(A)$, $g_i \in \g(A)$ for any $i$.

 There are distinct $1$-cocycles $\Lambda$ and $\Omega$ on $\G$ with coefficients in $L \GG_m$:
$$
\Lambda((h, \vp)) = h    \qquad \mbox{and}  \qquad \Omega( (h, \vp) )  = \frac{ d \vp(t)}{  dt} = \vp(t)'  \, \mbox{,}
$$
where $(h, \vp)  \in \G(A) = L \GG_m (A)  \rtimes \AutL (A)$.

Thus we obtain distinct $2$-cocycles on $\G$ with coefficients in $\GG_m$:
$$
\langle \Lambda , \Lambda   \rangle   \, \mbox{,} \qquad \quad \langle \Lambda , \Omega   \rangle
 \, \mbox{,} \qquad \quad  \langle \Omega , \Omega   \rangle  \, \mbox{.}
$$

\begin{Th} \label{last-th}
Consider a group ind-scheme which is a central extension of $\g$ by $\GG_m$ given  by a $2$-cocycle which is  either $\left\langle \Lambda , \Lambda   \right\rangle$, or $\langle \Lambda , \Omega   \rangle$,  or $\langle \Omega , \Omega   \rangle$. Let $a \in \GG_m(A)$, $h \in L \GG_m(A)$, $\vp \in \AutL (A)$.
\begin{enumerate}
\item The first group ind-scheme naturally acts on the line bundle $\langle  \mathfrak{F}, \mathfrak{F} \rangle$
on the moduli stack $\widetilde{{\mathcal M}}$ by isomorphisms $(a,(h, \vp))  \mapsto a \, T_{\langle \vp({\mathcal F}), \, \vp({\mathcal F}) \rangle }(h,h) \,  T_{\langle {\mathcal F}, \, {\mathcal F} \rangle }(\vp)  $.
\item The second group ind-scheme naturally acts on the line bundle $\langle  \mathfrak{F}, \bOmega \rangle$
on the moduli stack $\widetilde{{\mathcal M}}_{\rm sm}$ by isomorphisms $(a,(h, \vp))  \mapsto a \, T_{\langle \vp({\mathcal F}), \, \vp({\Omega_{C/A}}) \rangle }\left(h, \vp(t)' \right) \,  T_{\langle {\mathcal F}, \,{\Omega_{C/A}} \rangle}(\vp)  $.
\item \label{iitt-3} The third group ind-scheme naturally acts on the line bundle $\langle  \bOmega , \bOmega \rangle$
on the moduli stack $\widetilde{{\mathcal M}}_{\rm sm}$ by isomorphisms
$$(a,(h, \vp))  \longmapsto a \, T_{\langle \vp({\Omega_{C/A}}), \, \vp({\Omega_{C/A}}) \rangle }\left( \vp(t)'  , \vp(t)' \right) \,  T_{\langle {\Omega_{C/A}}, \,{\Omega_{C/A}} \rangle }(\vp)  \, \mbox{.}$$

Consider a group ind-scheme which is the pullback of the third central extension along the embedding  $\AutL \to \g$. This group ind-scheme  naturally acts on the line bundle $\langle  \bOmega , \bOmega \rangle$
on the moduli stack $\widetilde{{\mathcal M}}_{\rm curve}$ by isomorphisms
$$(a, \vp)  \longmapsto a \, T_{\langle \vp({\Omega_{C/A}}), \, \vp({\Omega_{C/A}}) \rangle }\left( \vp(t)'  , \vp(t)' \right) \,  T_{\langle {\Omega_{C/A}}, \,{\Omega_{C/A}} \rangle }(\vp)  \, \mbox{.}$$
\end{enumerate}

\end{Th}
\begin{proof}
Let $q = \left(C, p, t, \Omega_{C/A}, dt \right)$ be from $\widetilde{{\mathcal M}}_{\rm sm}(A)$.
For any $\vp$   from $ \AutL (A)$ we have the quintet $  \vp(q) = \left(\vp(C), \vp(p), \tilde{t}, \vp(\Omega_{C/A}), \widetilde{dt} \right)$  from
$\widetilde{{\mathcal M}}_{\rm sm}(A)$.
From diagrams~\eqref{cartesian}  and~\eqref{diag-sheaf}  and since the map $A((t))dt \to A((\vp(t))) d \vp(t)$ is the composition of
$A((t))dt \to A((\vp(t)))dt$ and $A((\vp(t)))dt \to  A((\vp(t))) d(\vp(t)) $, we have  a canonical isomorphism between quintets
 $$
   \vp(t)'   \left( \vp(q)  \right)
 \,
 \simeq  \,
 \left(\vp(C), \vp(p), \tilde{t}, \Omega_{\vp(C)/A} , d \phantom{} \tilde{t} \right) \, \mbox{,}$$
  compatible with base change.

Now the statemenets of the theorem, i.e. the calculations of the corresponding {$2$-cocycles}  on $\g$ (or on $\AutL$) via the actions on the corresponding line bundles on moduli stacks (see also~\eqref{coc-id}), follow from Proposition~\ref{aut-comp} and Theorem~\ref{Th-CC}.
\end{proof}

\quash{
(More accurate,  in~\cite[Prop.~(22)]{K} the morphism $\pi$ should be locally projective. But $\pi$ is \'{e}tale locally projective, because of existence of  ample invertible sheaves, see the proof of item~\ref{ii-last} of Theorem~\ref{Weil-gen}  and~\cite[Chapter~5, Theorem~1.34]{Liu}. Now we use the \'{e}tale descent.)
}

\vspace{0.3cm}

\noindent Steklov Mathematical Institute of Russsian Academy of Sciences, 8 Gubkina St., Moscow 119991, Russia,
 {\em and}

 \noindent National Research University Higher School of Economics, Laboratory of Mirror Symmetry,  6 Usacheva str., Moscow 119048, Russia,
{\em and}

\noindent National University of Science and Technology ``MISiS'',  Leninsky Prospekt 4, Moscow  119049, Russia

\noindent {\it E-mail:}  ${d}_{-} osipov@mi{-}ras.ru$

\end{document}